\documentclass[a4paper,11pt]{article}
\usepackage{amsmath,amssymb,amsfonts,amsthm}
\usepackage[top=30mm,bottom=30mm,left=25mm,right=25mm]{geometry}
\usepackage{setspace}
\usepackage{mathtools}
\usepackage{cite}
\usepackage{color}
\usepackage[pagewise]{lineno}
\usepackage{graphicx}
\usepackage{tikz}
\usetikzlibrary{patterns}
\usepackage{subcaption}
\newtheorem{thm}{Theorem}[section]

\newtheorem{lem}[thm]{Lemma}
\newtheorem{prop}[thm]{Proposition}

\newtheorem{rem}{Remark}

\theoremstyle{definition}
\newtheorem{defi}{Definition}[section]

\makeatletter
  
  \@addtoreset{equation}{section}
\makeatother

\newcommand{\RN}{{\mathbb{R}^N}}
\newcommand{\R}{\mathbb{R}}

\newcommand{\cL}{\mathcal{L}}

\DeclareMathOperator*{\esssup}{\rm{ess\,\sup}}

\title{Local existence and nonexistence of solutions to the Hardy parabolic equation with general nonlinearity}
\date{}
\author{
        Yo Tsusaka\footnote{
        E-mail address: tsusaka.yo.479@s.kyushu-u.ac.jp
        } \\
        {\small Joint Graduate School of Mathematics for Innovation, } \\
        {\small Kyushu University, } \\
        {\small Fukuoka, 819-0395, Japan} 
        }

\begin{document}
\maketitle
\renewcommand{\thefootnote}{\fnsymbol{footnote}}

\begin{abstract}
  In this paper, we consider the Cauchy problem for the Hardy parabolic equation with general nonlinearity and establish the local existence and nonexistence results. Our results provide the optimal integrability conditions on initial function for the existence of a local-in-time nonnegative solution. The proof of the existence result is based on the supersolution method. 
\end{abstract}


\section{Introduction}

We consider the existence and nonexistence of local-in-time solutions to the Cauchy problem for the Hardy parabolic equation with general nonlinearity:
\begin{equation} 
  \label{1.1} 
  \left\{ 
  \begin{array}{ll} 
  \partial_t u - \Delta u = |x|^{-\gamma} f(u), &x \in \RN, ~ t\in (0, T), \vspace{3pt} \\ 
  u(x,0)=u_0(x) \ge 0, &x \in \RN, 
  \end{array} 
  \right. 
\end{equation} 
where $N\ge 1$, $\partial_{t} = \partial / \partial t$, $\Delta = \sum_{i=1}^{N} \partial^2/\partial x_{i}^{2}$, $T > 0$, $f \in C([0, \infty))$, $0 < \gamma < \min\{2, N\}$ and $u_0$ is given nonnegative function. Throughout this paper, we assume that $f$ satisfies the following conditions: 
\begin{equation}\label{f}\tag{f}
  \begin{dcases}
    f\in C^1((0,\infty))\cap C([0,\infty)), &~ \\ 
    f(u)>0 &\text{for}~~u>0, \\
    f'(u)\ge 0 &\text{for}~~u>0, \\
    F(u)<\infty &\text{for}~~u>0, 
  \end{dcases}
\end{equation}
where
\begin{equation}\label{Fdef}
F(u):=\int_u^{\infty}\frac{d\tau}{f(\tau)}.
\end{equation}

The main purpose of this paper is to obtain a sufficient condition on the integrability of an initial function where problem $\eqref{1.1}$ admits a local-in-time nonnegative solution. We review the related results on the existence and nonexistence of local-in-time solutions to problem $\eqref{1.1}$. First, we review some results in the case $\gamma=0$, that is, the following Cauchy problem:
\begin{equation} 
  \label{nosing} 
  \left\{ 
  \begin{array}{lll} 
  \partial_t u - \Delta u = f(u), &x \in \RN, ~ t\in (0, T), \vspace{3pt} \\ 
  u(x,0)=u_0(x) \geq 0, &x \in \RN. 
  \end{array} 
  \right. 
\end{equation} 
One of the typical cases is a power type nonlinearity $f(u)=|u|^{p-1}u$ $(p>1)$. Then, Weissler~\cite{W80} obtained the following result. 

\begin{prop}\label{W80result}{\upshape{(Weissler~\cite{W80})}}
  Let $N\ge1$ and $f(u)=|u|^{p-1}u$ $(p>1)$. 
  \begin{itemize}
    \item[$\mathrm{(i)}$] $(\text{Existence})$ There exists $T>0$ such that $\eqref{nosing}$ admits a local-in-time solution $u\in C([0,T),L^r(\RN))$ if one of the following holds{\rm :}
    \begin{itemize}
      \item[$(a)$] $r>N(p-1)/2$, $r\ge1$ and $u_0\in L^r(\RN)$. 
      \item[$(b)$] $r=N(p-1)/2$, $r>1$ and $u_0\in L^r(\RN)$. 
    \end{itemize} 
    \item[$\mathrm{(ii)}$] $(\text{Nonexistence})$ For any $0<r<N(p-1)/2$, there exists $u_0\in L^r(\RN)$ such that $\eqref{nosing}$ admits no local-in-time solution. 
  \end{itemize}
\end{prop}

The exponent $N(p-1)/2$ arises from an invariance property of the equation under the following scaling{\rm :} 
\begin{equation}\label{powerscaling}
  u_{\lambda}(x,t)=\lambda^{\frac{2}{p-1}} u(\lambda x,\lambda^2 t), \quad \lambda>0. 
\end{equation}
If $u$ satisfies $\partial_t u-\Delta u=|u|^{p-1}u$, then $u_{\lambda}$ also satisfies the same equation. We see that for any $\lambda>0$, $\|u_{\lambda}(\cdot,0)\|_{r}=\|u(\cdot,0)\|_{r}$ if and only if $r=N(p-1)/2$. Therefore, Proposition~\ref{W80result} implies that the optimal integrability condition of $u_0$ for solvability of $\eqref{nosing}$ with $f(u)=|u|^{p-1}u$ is given by
\begin{equation*}
  u_0\in L^{\frac{N}{2}(p-1)}(\RN). 
\end{equation*}
Additionally, the solvability of the problem $\eqref{nosing}$ with $f(u)=|u|^{p-1}u$ $(p>1)$ has been extensively studied by many authors. See e.g. \cite{BC96, G86, LRSL16, W79, W80, W81} and references therein. 

We next review the results for $\eqref{nosing}$ with general case $f(u)$. To this end, we introduce some notation. For $1 \leq r \leq \infty$, we define uniformly local $L^r$ spaces $L^{r}_{\rm ul}(\RN)$ as
\begin{equation*}
L^r_{\rm ul}(\RN):=\left\{ u\in L^1_{\rm loc}(\RN)\left|\ \left\|u\right\|_{L^r_{\rm ul}(\RN)}<\infty\right.\right\}, 
\end{equation*}
where
\begin{equation*}
\left\|u\right\|_{L^r_{\rm ul}(\RN)}:=
\begin{dcases}
\sup_{y\in\RN}\left(\int_{B(y,1)}|u(x)|^r dx\right)^{\frac{1}{r}} & \textrm{if} \quad 1\le r<\infty, \\
\esssup_{y\in\RN} \left\| u \right\|_{L^{\infty}(B(y,1))} & \textrm{if} \quad r=\infty, 
\end{dcases}
\end{equation*}
with $B(y,\rho):=\{x\in\RN \mid  |x-y|<\rho\}$, for $\rho>0$. It is clear that $L^{\infty}_{\rm ul}(\RN) = L^{\infty}(\RN)$ and that $L^{\beta}_{\rm ul}(\RN) \subset L^{\alpha}_{\rm ul}(\RN)$ if $1\le \alpha \le \beta \le \infty$. 
We define $\cL^r_{\rm ul}(\RN)$ as
\begin{equation*}
\cL^r_{\rm ul}(\RN):=
\overline{BUC(\RN)}^{\|\,\cdot\,\|_{L^r_{\rm ul}(\RN)}},
\end{equation*}
namely, $\cL^r_{\rm ul}(\RN)$ denotes the closure of the space of bounded uniformly continuous functions $BUC(\RN)$ in the space $L^r_{\rm ul}(\RN)$. We denote $\|\cdot\|_{L^{r}(\RN)}$ by $\|\cdot\|_{r}$ for $1\le r\le\infty$. For $q\in\R$, we define $S_{q}$ as
\begin{equation*}
S_{q}:=\left\{
f\in C([0,\infty))\left|\ \textrm{$f$ satisfies (\ref{f}) and the limit $q:=\lim_{\xi\to\infty}f'(\xi)F(\xi)$ exists.} \right\} \right..
\end{equation*}
It is known that $S_{q} = \emptyset$ if $q<1$ (see \cite[Lemma~2.1]{FI23}, for instance). 

We are now ready to introduce the result on local existence of solutions to problem $\eqref{nosing}$. Fujishima--Ioku \cite{FI18} and Miyamoto--Suzuki \cite{MS21} obtained the following result. 
\begin{prop}\label{FIMSresult}{\upshape{(Fujishima--Ioku~\cite{FI18}, Miyamoto--Suzuki~\cite{MS21})}}
  \begin{itemize}
    \item[$\mathrm{(i)}$] $(\text{Existence})$ Let $u_0 \ge 0$ and $f \in S_{q}$ with $q \in [1, \infty)$. Then, there exists $T>0$ such that $\eqref{nosing}$ admits a local-in-time nonnegative solution $u$ in $(0,T)$, if one of the following holds$:$ 
    \begin{itemize}
      \item[$(a)$] $r>N/2$, $q<1+r$ and $F(u_0)^{-r} \in L^{1}_{\rm ul}(\RN)$. 
      \item[$(b)$] $r>N/2$, $q=1+r$, $F(u_0)^{-r} \in L^{1}_{\rm ul}(\RN)$ and 
      \begin{equation}\label{sub-assum}
        f'(\xi) F(\xi) \le q \quad \text{for large} ~~ \xi>0. 
      \end{equation}
      \item[$(c)$] $r=N/2$, $q<1+r$ and $F(u_0)^{-r} \in \cL^{1}_{\rm ul}(\RN)$. 
    \end{itemize}
    \item[$\mathrm{(ii)}$] $(\text{Nonexistence})$ Suppose that $f \in C^{2}([0, \infty)) \cap S_{q}$ with $q \in [1, \infty)$ and $f$ is convex in $[0, \infty)$. Let $0<r<N/2$ and $q\le1+r$. Then, there exists a nonnegative initial function $u_0$ such that $F(u_0)^{-r} \in L^{1}_{\rm ul}(\RN)$ and $\eqref{1.1}$ admits no nonnegative solution. 
  \end{itemize}
\end{prop}
\begin{rem}\label{Remark1}
  \begin{itemize}
    \item[$\mathrm{(i)}$] Proposition~$\ref{FIMSresult}$ shows that the optimal integrability condition of initial function $u_0$ for solvability of $\eqref{nosing}$ is given by
    \begin{equation}\label{nosing-oic}
      F(u_0)^{-\frac{N}{2}} \in \cL^{1}_{\rm ul}(\RN). 
    \end{equation}
    \item[$\mathrm{(ii)}$] The integrability condition $\eqref{nosing-oic}$ is derived from the following generalized scaling proposed by Fujishima~{\rm \cite{F14}:} 
    \begin{equation}\label{nosingscaling}
      u_{\lambda}(x,t) := F^{-1}[\lambda^{-2} F(u(\lambda x, \lambda^{2} t))], \quad \lambda>0, 
    \end{equation}
    where $F^{-1}$ is an inverse function of $F$. Note that if $u$ satisfies $\eqref{nosing}$, then $u_{\lambda}$ satisfies 
    \begin{equation}\label{nosinginveq}
      \partial_{t} u_{\lambda} = \Delta u_{\lambda} + f(u_{\lambda}) + \frac{|\nabla u_{\lambda}|^{2}}{f(u_{\lambda}) F(u_{\lambda})} (f'(u) F(u) - f'(u_{\lambda}) F(u_{\lambda})). 
    \end{equation}
    Hence, generalized scaling $\eqref{nosingscaling}$ does not necessarily preserve the equation $\eqref{nosing}$. However, the following invariance holds{\rm :} 
    \begin{equation*}
      \int_{\RN} \frac{1}{F(u(x, 0))^{\frac{N}{2}}} \,dx = \int_{\RN} \frac{1}{F(u_{\lambda}(x, 0))^{\frac{N}{2}}} \,dx, \quad \lambda>0. 
    \end{equation*}
    This plays a crucial role for classifying the existence and nonexistence of a local-in-time solution to $\eqref{nosing}$ in Lebesgue space, see {\rm \cite{FI18}} for detail. 
    \item[$\mathrm{(iii)}$] When $f(u)=u^p$ $(p>1)$, then $F(u)=(p-1)^{-1} u^{-(p-1)}$ and $f'(u)F(u)=p/(p-1)$. Thus, the scaling $\eqref{nosingscaling}$ coincides with $\eqref{powerscaling}$ and the remainder term of the equation $\eqref{nosinginveq}$ vanishes. Moreover, since $F(u)^{-N/2}=(p-1)^{N/2} u^{N(p-1)/2}$, Proposition~$\ref{FIMSresult}$ is a generalization of Proposition~$\ref{W80result}$, 
  \end{itemize}
\end{rem}
Furthermore, there are various results on the local and global solvability of $\eqref{nosing}$. See e.g. \cite{FHIL24, FI18, FI21, FI23, GM22, HM25, MS21} and references therein. 

We move to review in the case $\gamma>0$. In the typical case $f(u)=|u|^{p-1}u$ $p>1$, there are various results on the solvability. Wang \cite{W93} proved the local existence of solutions in $C([0, T]; C_{B}(\RN))$ for all $u_0 \in C_{B}(\RN)$, where $C_{B}(\RN)$ is the space of bounded continuous functions. Ben Slimene--Tayachi--Weissler \cite{BTW17} proved well-posedness in $C_{0}(\RN)$ and Lebesgue space $L^{r}(\RN)$ for a suitable $r>1$, where $C_{0}(\RN)$ is the space of continuous functions vanishing at infinity. Furthermore, Chikami--Ikeda--Taniguchi \cite{CIT22} proved well-posedness in weighted Lebesgue spaces $L^{r}_{s}(\RN)$ for suitable $r>0$ and $s\in\R$. In those results, the proof is based on the contraction mapping principle in suitable function spaces. Additionally, problem $\eqref{1.1}$ with $f(u)=|u|^{p-1}u$ is also studied by many authors from the point of view of the global solvability and asymptotic behavior. See e.g. \cite{C19, CIT21, CIT22, HS24, HT21, T20} and references therein. 

  For general case $f(u)$, Castillo--Guzm\'{a}n-Rea--Loayza \cite{CGL22} and Carhuas-Torre--Castillo--Loayza \cite{CCL25} studied the local existence in Lebesgue spaces and uniformly local Lebesgue spaces, respectively. In particular, it was shown in \cite{CCL25} that $\eqref{1.1}$ admits a local-in-time nonnegative solution for every nonnegative initial function $u_0\in L^{r}_{\rm ul}(\RN)$ if and only if $f$ satisfies
\begin{equation}\label{CCL25NS}
  \begin{dcases}
    \int_{1}^{\infty} \tau^{-\left(1+\frac{2-\gamma}{N} \right)} \tilde{F}(\tau) \,d\tau<\infty &\text{if}~~r=1, \\
    \limsup_{\xi\to\infty} \xi^{-1+\frac{(2-\gamma)r}{N}} f(\xi) <\infty, &\text{if}~~r>1, 
  \end{dcases}
\end{equation}
where $\tilde{F}(\tau) = \sup_{1\le\sigma\le\tau} f(\sigma)/\sigma$, $\tau>0$. This implies that the condition $\eqref{CCL25NS}$ is necessary and sufficient conditions on $f$ for which $\eqref{1.1}$ admits a local-in-time solution for every initial function $u_0$ belonging to uniformly local Lebesgue space. On the other hand, to the best of our knowledge, there are no results on the conditions of initial function $u_0$ for which $\eqref{1.1}$ admits a local-in-time solution for given $f$. In this paper, we study this problem and give sufficient conditions of integrability on $u_0$. 

Before we state our main results, as in Remark~1~(ii), we define the generalized scaling for $\eqref{1.1}$: for $\lambda>0$, 
\begin{equation*}
  u_{\lambda}(x, t) := F^{-1}[\lambda^{-(2-\gamma)} F(u(\lambda x, \lambda^{2} t))]. 
\end{equation*}
Then, we see that if $u$ is a solution to $\eqref{1.1}$, then for all $\lambda>0$, $u_{\lambda}$ satisfies
\begin{equation*}
  \partial_{t} u_{\lambda} = \Delta u_{\lambda} + |x|^{-\gamma} f(u_{\lambda}) + \frac{|\nabla u_{\lambda}|^{2}}{f(u_{\lambda}) F(u_{\lambda})} (f'(u) F(u) - f'(u_{\lambda}) F(u_{\lambda})). 
\end{equation*}
Furthermore, the following invariance holds: 
\begin{equation*}
  \int_{\RN} \frac{1}{F(u(x, 0))^{\frac{N}{2-\gamma}}} \,dx = \int_{\RN} \frac{1}{F(u_{\lambda}(x, 0))^{\frac{N}{2-\gamma}}} \,dx, \quad \lambda>0. 
\end{equation*}
By this invariance, we expect that the optimal integrability condition of initial function $u_0$ for local existence of solutions to $\eqref{1.1}$ is given by
\begin{equation}\label{OIC}
  F(u_0)^{-\frac{N}{2-\gamma}} \in \cL^{1}_{\rm ul}(\RN). 
\end{equation}

Now, we introduce the definition of solution to $\eqref{1.1}$. Let $\{e^{t\Delta}\}_{t>0}$ be a heat semi-group on $\RN$;
\begin{equation*}
  [e^{t\Delta} \phi](x) := \int_{\RN} G(x-y, t) \phi(x) \,dy \quad \text{for}~~\phi \in L^{1}_{\rm ul}(\RN), 
\end{equation*}
where $G(x, t) = (4\pi t)^{-N/2} \exp{(-|x|^{2} / 4t)}$. We define a solution to $\eqref{1.1}$. 

\begin{defi}\label{Defsol}
  We say that a nonnegative measurable function $u$ is a local-in-time solution to \eqref{1.1} in $\RN\times (0, T)$ if there exists $T>0$ such that $u \in L^{\infty}((0, T), L^{1}_{\rm ul}(\RN)) \cap L^{\infty}_{\rm loc}((0, T), L^{\infty}(\RN))$ and $u$ satisfies 
  \begin{equation*}
      u(x, t) = [e^{t\Delta} u_0](x) + \int_{0}^{t} [e^{(t-s)\Delta} |\cdot|^{-\gamma} f(u(\cdot, s))](x) \,ds
    \end{equation*}
    for a.a. $x\in \RN$, $0<t<T$. 
    We say that a nonnegative measurable function $\overline{u}$ is a supersolution of $\eqref{1.1}$ if there exists $T>0$ such that 
    \begin{equation*}
      \overline{u}(x, t) \ge [e^{t\Delta} u_0](x) + \int_{0}^{t} [e^{(t-s)\Delta} |\cdot|^{-\gamma} f(\overline{u}(\cdot, s))](x) \,ds
    \end{equation*}
    for a.a. $x \in \RN$, $0 < t < T$. 
\end{defi}

We are ready to state the main theorem on the local existence of $\eqref{1.1}$. 

\begin{thm}\label{Exist-thm}
  Let $N\ge1$, $0<\gamma <\min\{2, N\}$, $u_0\ge 0$ and $f \in S_{q}$ with $q \in [1, \infty)$. Then, there exists $T>0$ such that $\eqref{1.1}$ has a local-in-time nonnegative solution $u$ in $\RN \times(0,T)$, if one of the following holds{\rm :} 
  \begin{itemize}
    \item[$\mathrm{(i)}$] $r>N/(2-\gamma)$, $q<1+r$ and $F(u_0)^{-r} \in L^{1}_{\rm ul}(\RN)$. 
    \item[$\mathrm{(ii)}$]  $r>N/(2-\gamma)$, $q=1+r$, $F(u_0)^{-r} \in L^{1}_{\rm ul}(\RN)$ and $\eqref{sub-assum}$ holds.   
    \item[$\mathrm{(iii)}$] $r=N/(2-\gamma)$, $q<1+r$ and $F(u_0)^{-r} \in \cL^{1}_{\rm ul}(\RN)$. 
  \end{itemize}
  Moreover, in all cases, there exists a constant $C>0$ depending on $N$ and $u_0$ such that
  \begin{equation}\label{Exist-thm-est}
    \|F(u(t))^{-r}\|_{L^{1}_{\rm ul}(\RN)} \le C \quad \text{for}~~0<t<T. 
  \end{equation}
\end{thm}

\begin{rem}
  \begin{itemize}
    \item[$\mathrm{(i)}$] By Theorem~$\ref{Exist-thm}$, we obtain the sufficient conditions with respect to $r$ and $q$ for local existence of $\eqref{1.1}$ $(\text{see {\bf Fig.~\ref{q-rplane}}})$. Note that Theorem~$\ref{Exist-thm}$ does not cover the doubly critical case $q=1+r$ and $r=N/(2-\gamma)$. 
    \item[$\mathrm{(ii)}$] In the cases $q<1+r$ or $q=1+r$ with $\eqref{sub-assum}$, $F(\xi)^{-r}$ is convex for large $\xi>0$. Thus, if $F(u_0)^{-r}\in L^{1}_{\rm ul}(\RN)$, then $u_0\in L^{1}_{\rm ul}(\RN)$ and hence, $e^{t\Delta}u_0<\infty$ for $t>0$. 
  \end{itemize}
\end{rem}

We prove Theorem~\ref{Exist-thm} by so-called supersolution method (see Proposition~\ref{equisuper}). In particular, we employ the argument by Miyamoto--Suzuki \cite[Theorem~A]{MS21} for $\eqref{nosing}$. In our case, we need to take care of the term $|\cdot|^{-\gamma}$ in $\eqref{1.1}$. To handle this, we use H\"{o}lder's inequality and Jensen's inequality with some estimate of the heat kernel, and then we obtain the estimate of the integral term of a supersolution, see a proof of Lemma~\ref{3lem1} for detail. Furthermore, in the critical case (iii), we derive the estimate of $f(F^{-1}(\cdot))$ (see Lemma~\ref{keylemofkeyprop}) and show the convergence of $\eqref{3.2}$. 

Next, we state the nonexistence results of local-in-time solutions to $\eqref{1.1}$. 

\begin{thm}\label{Nonexist-thm}
  Suppose that $f \in C^{2}([0, \infty)) \cap S_{q}$ with $q \in [1, \infty)$ and 
  \begin{equation}\label{Nonexist-assum}
    \frac{f(\xi) f''(\xi)}{(f'(\xi))^{2}}\ge\delta \quad \text{for}~~\xi\ge 0
  \end{equation}
  for sufficiently small $\delta>0$. Let $0<r<N/(2-\gamma)$ and $q \le 1+r$. Then, there exists a nonnegative initial function $u_0 \in L^{1}_{\rm ul}(\RN)$ such that $F(u_0)^{-r} \in L^{1}_{\rm ul}(\RN)$ and $\eqref{1.1}$ does not have nonnegative local-in-time solutions. 
\end{thm}

\begin{figure}[t]
  \centering
  \begin{tikzpicture}[scale=1.0, >=stealth]
  \draw[->, semithick] (-0.5,0) -- (6.0,0) node[right] {$q$};
  \draw[->, semithick] (0,-0.5) -- (0,5) node[above] {$r$};
  \node[below left] at (0,0) {$O$};
  \def\Na{2.5}    
  \def\Nb{3}      
  \def\Qa{1}      
  \def\Qb{1+\Na}  
  \coordinate (A) at (0,0);
  \coordinate (B) at (0,\Na);
  \coordinate (C) at (0,\Nb);
  \coordinate (D) at (\Qa,0);
  \coordinate (E) at (\Qa,\Na);
  \coordinate (F) at (\Qa,\Nb);
  \coordinate (G) at (\Qa,4.8);
  \coordinate (I) at (\Qb,0);
  \coordinate (J) at (\Qb,\Na);
  \coordinate (K) at (5.8,4.8);
  \draw[dotted, semithick] (B)--(E);   
  \draw[dotted, semithick] (I)--(J);   
  \fill[pattern=north west lines, pattern color=black!100, ]
    (E)--(F)--(G)--(K)--(J)--(E)--cycle;
  \fill[gray!50] (D)--(E)--(J)--(D)--cycle;
  \draw[->, thin] (2.7,0.7) -- (1.9,1.6);
  \draw[->, thin] (4.2,\Na) -- (4.2,3.1);
  \draw[->, thin] (3.7,1.4) -- (2.6,1.4) -- (2.6,2.4);
  \draw[-> , opacity=0] (1.25,-0.7) -- (1.25,-0.1);
  \draw[semithick] (D)--(G);       
  \draw[ultra thick] (E)--(J);     
  \draw[semithick] (D)--(K);       
  \draw[ultra thick] (J)--(K);     
  \node[fill=white, inner sep=2pt] at (2.5,3.7) {Existence~(i)};
  \node[fill=white, inner sep=2pt] at (5.0,2.3) {Existence~(ii)};
  \node[fill=white, inner sep=2pt] at (5.0,1.4) {Existence~(iii)};
  \node[fill=white, inner sep=2pt] at (3.2,0.6) {Nonexistence};
  \node[right] at (4.6,3.3) {$r = q - 1$};
  \node[below] at (1,0) {$1$};
  \node[below] at (\Qb,0) {$1 + \dfrac{N}{2 - \gamma}$};
  \node[left] at (0,\Na) {$\dfrac{N}{2 - \gamma}$};
  \node[opacity=0.0,below] at (1.25,-0.7) {$\dfrac{\gamma}{2 - \gamma}$};
\end{tikzpicture}
  \caption{Existence and nonexistence area of a local-in-time solution in Theorem~\ref{Exist-thm} and Theorem~\ref{Nonexist-thm} . Existence~(i), (ii) and (iii) correspond to the results of Theorem~\ref{Exist-thm} (i), (ii) and (iii), respectively.}
  \label{q-rplane}
\end{figure}
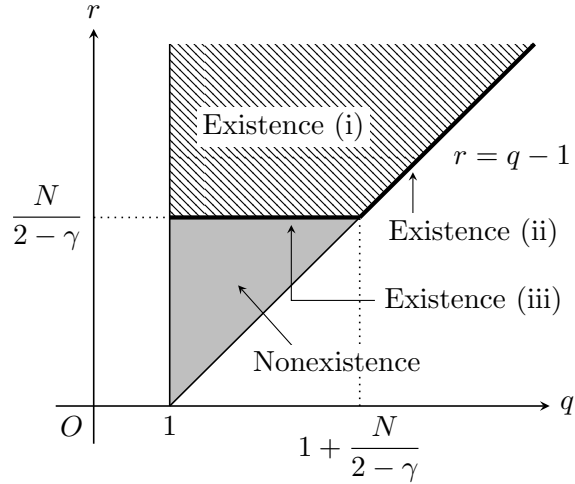

\begin{rem}
  \begin{itemize}
    \item[$\mathrm{(i)}$] By Theorem~$\ref{Nonexist-thm}$, we obtain the conditions on $r$ and $q$ for the nonexistence of local-in-time solutions to $\eqref{1.1}$ $(\text{see {\bf Fig.~\ref{q-rplane}}})$.  
    \item[$\mathrm{(ii)}$] Theorem~$\ref{Exist-thm}$ and Theorem~$\ref{Nonexist-thm}$ imply that $\eqref{OIC}$ is the optimal integrability condition of initial function $u_0$ for local solvability of $\eqref{1.1}$. 
    \item[$\mathrm{(iii)}$] The assumption $\eqref{Nonexist-assum}$ is important for our approach because we use the convexity of $f^{1-\delta}$. We note that $\eqref{Nonexist-assum}$ holds if $f(\xi) = \xi^{p}$ $(p>1)$ or $f(\xi) = e^{\xi}$. 
  \end{itemize}
\end{rem}

We prove Theorem~\ref{Nonexist-thm} by contradiction which is based on the argument in \cite{ACCL25}, \cite{FI21} and \cite{MS21} etc. for $\eqref{nosing}$. Similar to Theorem~\ref{Exist-thm}, we need to take care of the term $|\cdot|^{-\gamma}$ in $\eqref{1.1}$. To take care of them, we use the condition $\eqref{Nonexist-assum}$ and derive the reverse H\"{o}lder's inequality. Then, we obtain the necessary condition for the existence of local-in-time nonnegative solutions to $\eqref{1.1}$ (see Lemma~\ref{Nonexist-lem}). By using this necessary condition, we prove Theorem~\ref{Nonexist-thm} by the contradiction argument. 

The outline of the paper is as follows. In Section~2, we introduce several properties of uniformly local Lebesgue space $L^{r}_{\rm ul}(\RN)$. We also give a pointwise estimate of $e^{(t-s)\Delta}[\left|\cdot\right|^{-\alpha}e^{s\Delta}\phi]$, which is necessary to control the nonlinear term. In Section~3, we introduce an abstract existence result for $\eqref{1.1}$ and give a sufficient condition on the existence of local-in-time solution to problem $\eqref{1.1}$ (see Proposition~\ref{MainProp}). As an application of Proposition~\ref{MainProp}, we prove Theorem~\ref{Exist-thm}. In Section~4, we show some lemma needed to prove Theorem~\ref{Nonexist-thm} and show a proof of Theorem~\ref{Nonexist-thm}. 

Throughout this paper, we will use $C$ to denote generic positive constants and $C$ may take different values within a calculation.


\section{Preliminaries}

First, we recall some properties on functions in uniformly local Lebesgue space. 

\begin{lem}[\text{\cite[Lemma~2.8]{MS21}}]\label{Lulmax}
  Let $A>0$ be a constant and $\phi \in \cL^{1}_{\rm ul}(\RN)$. Then, $\max\{\phi, A\} \in \cL^{1}_{\rm ul}(\RN)$. 
\end{lem}

We introduce estimates for heat semi-group in the space $L^{r}_{\rm ul}(\RN)$. 

\begin{prop}\label{LpLqul}
  \begin{itemize}
    \item[$\mathrm{(i)}$] Let $N\ge 1$ and $1\le p \le q \le \infty$. Then, there exists a constant $C>0$ such that, for $\phi \in L^{p}_{\rm ul}(\RN)$, 
    \begin{equation*}
      \|e^{t\Delta} \phi\|_{L^{q}_{\rm ul}(\RN)} \le C \left(1+t^{-\frac{N}{2} \left(\frac{1}{p} - \frac{1}{q} \right)} \right) \|\phi\|_{L^{p}_{\rm ul}(\RN)} \quad \text{for all }~t>0. 
    \end{equation*}
    In particular, 
    \begin{equation*}
      \|e^{t\Delta} \phi\|_{L^{q}_{\rm ul}(\RN)} \le C t^{-\frac{N}{2} \left(\frac{1}{p} - \frac{1}{q} \right)} \|\phi\|_{L^{p}_{\rm ul}(\RN)} \quad \text{for }~0<t\le 1. 
    \end{equation*}
    \item[$\mathrm{(ii)}$] Let $N\ge 1$ and $1\le p < q \le \infty$. Then, for any $\phi \in \cL^{p}_{\rm ul}(\RN)$ and $\varepsilon > 0$, there exists $t_{\varepsilon}>0$ such that 
    \begin{equation*}
      \|e^{t\Delta} \phi\|_{L^{q}_{\rm ul}(\RN)} \le \varepsilon t^{-\frac{N}{2} \left(\frac{1}{p} - \frac{1}{q} \right)} \quad \text{for}~~0<t<t_{\varepsilon}. 
    \end{equation*}
  \end{itemize}
\end{prop}

\begin{proof}
  For a proof of Proposition \ref{LpLqul}, see \cite[Proposition 2.1]{ARCD04}, \cite[Corollary 3.1]{MT06} and \cite[Proposition 2.4 and 2.5]{GM22}. 
\end{proof}

\begin{prop}\label{Jensen}
  Let $\zeta\ge 0$ be a constant. 
  \begin{itemize}
    \item[$\mathrm{(i)}$] Suppose that $H:[\zeta, \infty) \to [0, \infty)$ is a convex function. If $\phi \in L^{1}_{\rm ul}(\RN)$, $H(\phi) \in L^{1}_{\rm ul}(\RN)$ and $\phi \ge \zeta$ in $\RN$, then
    \begin{equation*}
      H(e^{t\Delta} \phi) \le e^{t\Delta} H(\phi) \quad \text{in}~~\RN \times (0, \infty). 
    \end{equation*}
    \item[$\mathrm{(ii)}$] Suppose that $K:[\zeta, \infty) \to [0, \infty)$ is a concave function. If $\phi \in L^{1}_{\rm ul}(\RN)$ and $\phi \ge \zeta$ in $\RN$, then
    \begin{equation*}
      K(e^{t\Delta} \phi) \ge e^{t\Delta} K(\phi) \quad \text{in}~~\RN \times (0, \infty). 
    \end{equation*}
  \end{itemize}
\end{prop}

\begin{proof}
  It follows from Jensen's inequality. See \cite[Proposition~2.9]{GM22} for a proof. 
\end{proof}

Next, we give a key lemma to prove Theorem~\ref{Exist-thm}. See \cite[Lemma~4.1]{CGL22}. 

\begin{lem}[\text{\cite[Lemma~4.1]{CGL22}}]\label{keylem1}
  Suppose that $\phi \in L^{1}_{\rm ul}(\RN)$, $\phi \ge 0$ and $0 < \alpha < N$. Then, there exists a constant $C=C(N,\gamma)>0$ such that
  \begin{equation*}
    e^{t\Delta}\left[|\cdot|^{–\alpha} e^{s\Delta} \phi \right] \le C \left(\frac{1}{t} + \frac{1}{s} \right)^{\frac{\alpha}{2}} e^{(t+s)\Delta} \phi\quad \text{for}~~t, s > 0. 
  \end{equation*}
\end{lem}

We next give an estimate of $f(F^{-1}(\cdot))$. 

\begin{lem}\label{keylemofkeyprop}
  Let $f \in S_{q}$ with $q \in [1, \infty)$. Then, for any $\varepsilon>0$, there exists $\delta_{\varepsilon}>0$ such that 
  \begin{equation*}
    \rho^{-(q-\varepsilon)} f\left(F^{-1}(A) \right) \le f\left(F^{-1}(A\rho) \right) \le \rho^{-(q+\varepsilon)} f\left(F^{-1}(A) \right)
  \end{equation*}
  for $A \in (0, \delta_{\varepsilon})$ and $\rho\in (0, 1)$. 
\end{lem}

\begin{proof}
  Fix $\varepsilon>0$ arbitrarily and we define 
  \begin{equation*}
    h(\xi) := f\left(F^{-1}(\xi) \right), \quad \xi>0. 
  \end{equation*}
  Since
  \begin{equation*}
    h'(\xi) = -f'\left(F^{-1}(\xi)\right) f\left(F^{-1}(\xi)\right) = -f'\left(F^{-1}(\xi)\right) h(\xi)
  \end{equation*}
  for $\xi>0$, we have 
  \begin{equation*}
    \frac{\xi h'(\xi)}{h(\xi)} = -\xi f'\left(F^{-1}(\xi)\right) = -f'\left(F^{-1}(\xi)\right) F\left(F^{-1}(\xi)\right)
  \end{equation*}
  for $\xi>0$. Then, we obtain
  \begin{equation*}
    \lim_{\xi\to0} \frac{\xi h'(\xi)}{h(\xi)} = -q. 
  \end{equation*}
  Thus, there exists $\delta_{\varepsilon}>0$ such that
  \begin{equation*}
    \frac{q-\varepsilon}{\xi} \le -\frac{h'(\xi)}{h(\xi)} \le \frac{q+\varepsilon}{\xi}
  \end{equation*}
  for $\xi \in (0, \delta_{\varepsilon})$. Integrating on $(A\rho, A)$ for $A\in (0, \delta_{\varepsilon})$ and $\rho\in(0, 1)$, we have
  \begin{equation*}
    \log{\rho^{-(q-\varepsilon)}} \le \log{\frac{h(A\rho)}{h(A)}} \le \log{\rho^{-(q+\varepsilon)}}, 
  \end{equation*}
  that is, 
  \begin{equation*}
    \rho^{-(q-\varepsilon)} h(A) \le h(A\rho) \le \rho^{-(q+\varepsilon)} h(A). 
  \end{equation*}
  By the definition of $h$, we have
  \begin{equation*}
    \rho^{-(q-\varepsilon)} f\left(F^{-1}(A) \right) \le f\left(F^{-1}(A\rho)\right) \le \rho^{-(q+\varepsilon)} f\left(F^{-1}(A)\right)
  \end{equation*}
  for $A\in (0, \delta_{\varepsilon})$ and $\rho \in (0, 1)$. Thus, this completes the proof of Lemma \ref{keylemofkeyprop}. 
\end{proof}

We next prove the following proposition which is a crucial estimate for a proof of Proposition~\ref{MainProp} (ii). 

\begin{prop}\label{keyprop}
  Let $N\ge 1$ and $0<\gamma<\min\{2, N\}$. Let $f \in S_{q}$ with $q \in [1, \infty)$ and $\theta>(2-\gamma)(q-1)/N$. Then, 
  \begin{equation*}
    \limsup_{\eta\to\infty} \frac{1}{f(\eta)} \int_{0}^{1} \sigma^{\frac{N}{2} \left(\theta+\frac{2-\gamma}{N} \right)-1} (1-\sigma)^{-\frac{\gamma}{2}} f\left(F^{-1}\left(F(\eta) \sigma^{\frac{2-\gamma}{2}} \right) \right) \,d\sigma < \infty. 
  \end{equation*}
\end{prop}

\begin{proof}
  Let $\varepsilon>0$ satisfy 
  \begin{equation*}
    \frac{N}{2-\gamma} \theta + 1 > q+\varepsilon. 
  \end{equation*}
  By Lemma \ref{keylemofkeyprop}, there exists $\delta_{\varepsilon}>0$ such that 
  \begin{equation*}
    f\left(F^{-1}(F(\eta) \sigma^{\frac{2-\gamma}{\gamma}}) \right) \le \sigma^{-\frac{(2-\gamma)}{2} (q+\varepsilon)} f(\eta)
  \end{equation*}
  for $\sigma \in (0,1)$ and $\eta>0$ with $F(\eta) \in (0, \delta_{\varepsilon})$. Then, we have
  \begin{align*}
    \frac{1}{f(\eta)} \int_{0}^{1} &\sigma^{\frac{N}{2} \left(\theta+\frac{2-\gamma}{N} \right)-1} (1-\sigma)^{-\frac{\gamma}{2}} f\left(F^{-1}\left(F(\eta) \sigma^{\frac{2-\gamma}{2}} \right) \right) \,d\sigma \\
    &\le \int_{0}^{1} \sigma^{\frac{N}{2} \left(\theta+\frac{2-\gamma}{N} \right) -\frac{(2-\gamma)}{2} (q+\varepsilon)-1} (1-\sigma)^{-\frac{\gamma}{2}}  \,d\sigma 
  \end{align*}
  for $\eta \in (F^{-1}(\delta_{\varepsilon}), \infty)$. Since $\gamma/2 < 1$ and
  \begin{equation*}
    \frac{N}{2} \left(\theta+\frac{2-\gamma}{N} \right) -\frac{(2-\gamma)}{2} (q+\varepsilon) > \frac{N}{2} \theta +\frac{2-\gamma}{2} -\frac{(2-\gamma)}{2} \left(\frac{N}{2-\gamma} \theta + 1 \right) = 0, 
  \end{equation*}
  we obtain
  \begin{equation*}
    \limsup_{\eta\to\infty} \frac{1}{f(\eta)} \int_{0}^{1} \sigma^{\frac{N}{2} \left(\theta+\frac{2-\gamma}{N} \right)-1} (1-\sigma)^{-\frac{\gamma}{2}} f\left(F^{-1}\left(F(\eta) \sigma^{\frac{2-\gamma}{2}} \right) \right) \,d\sigma<\infty. 
  \end{equation*}
  This completes the proof of Proposition~\ref{keyprop}. 
\end{proof}

We give an estimate of $F$ and $F^{-1}$ needed to show Theorem~\ref{Nonexist-thm}. 

\begin{lem}\label{estofF}
  Let $f \in S_{q}$ with $q\in [1, \infty)$. Then, for any $\varepsilon>0$, there exist $\zeta_{\varepsilon}>0$, $\sigma_{\varepsilon}>0$ and $C=C_{\varepsilon}>0$ such that
  \begin{equation*}
    F(\xi) \le C \xi^{-\frac{1}{q-1+\varepsilon}}\quad \text{for}~~\xi\ge\zeta_{\varepsilon}, 
  \end{equation*}
  and
  \begin{equation*}
    F^{-1}(\sigma) \le C \sigma^{-(q-1+\varepsilon)} \quad \text{for}~~\sigma\in(0, \sigma_{\varepsilon}). 
  \end{equation*}
\end{lem}

\begin{proof}
  Let $\varepsilon>0$. Since $f'(\xi) F(\xi) \to q$ as $\xi\to\infty$, there exists $\zeta_{\varepsilon}>0$ such that
  \begin{equation*}
    f'(\xi) F(\xi) \le q+\varepsilon \quad \text{for}~~\xi\ge\zeta_{\varepsilon}. 
  \end{equation*}
  Since $F''(\xi) = f'(\xi) (F'(\xi))^{2}$ and $F'(\xi) < 0$ for $\xi\ge0$, we have
  \begin{equation*}
    \frac{F''(\xi)F(\xi)}{(F'(\xi))^{2}} \le q+\varepsilon \quad \text{for}~~\xi\ge\zeta_{\varepsilon}, 
  \end{equation*}
  that is, 
  \begin{equation*}
    (q+\varepsilon) \frac{F'(\xi)}{F(\xi)} \le \frac{F''(\xi)}{F'(\xi)} \quad \text{for}~~\xi\ge\zeta_{\varepsilon}. 
  \end{equation*}
  Integrating this over $[\zeta_{\varepsilon}, \xi]$ for $\xi > \zeta_{\varepsilon}$, we have
  \begin{equation*}
    \frac{F'(\zeta_{\varepsilon})}{F'(\xi)} \le \left(\frac{F(\zeta_{\varepsilon})}{F(\xi)} \right)^{q+\varepsilon}, 
  \end{equation*}
  that is, 
  \begin{equation*}
    -\frac{F'(\zeta_{\varepsilon})}{F(\zeta_{\varepsilon})^{q+\varepsilon}} F(\xi)^{q+\varepsilon} \le -F'(\xi). 
  \end{equation*}
  Thus, we have
  \begin{equation*}
    (q-1+\varepsilon) C \le \frac{d}{d\xi} \left(F(\xi)^{-(q-1+\varepsilon)} \right) \quad \text{for}~~\xi >\zeta_{\varepsilon}. 
  \end{equation*}
  Integrating this over $(\zeta_{\varepsilon}, \xi)$ for $\xi > \zeta_{\varepsilon}$, we have
  \begin{equation*}
    (q-1+\varepsilon) C (\xi - \zeta_{\varepsilon}) \le F(\xi)^{-(q-1+\varepsilon)} - F(\zeta_{\varepsilon})^{-(q-1+\varepsilon)}, 
  \end{equation*}
  that is, 
  \begin{equation*}
    F(\xi)^{-(q-1+\varepsilon)} \ge (q-1+\varepsilon) C \xi + C\quad \text{for}~~\xi >\zeta_{\varepsilon}. 
  \end{equation*}
  Hence, 
  \begin{equation*}
    F(\xi)^{-(q-1+\varepsilon)} \ge C \xi \quad \text{for}~~\xi>\zeta_{\varepsilon}. 
  \end{equation*}
  Therefore, we have
  \begin{equation*}
    F(\xi) \le C \xi^{-\frac{1}{q-1+\varepsilon}} \quad \text{for}~~\xi>\zeta_{\varepsilon}. 
  \end{equation*}
  Furthermore, since $F^{-1}$ is decreasing, we have
  \begin{equation*}
    F^{-1}\left(C \xi^{-\frac{1}{q-1+\varepsilon}} \right) \le \xi, 
  \end{equation*}
  and hence there exist $\sigma_{\varepsilon}>0$ such that
  \begin{equation*}
    F^{-1}(\sigma) \le C \sigma^{-(q-1+\varepsilon)}
  \end{equation*}
  for $\sigma \in (0, \sigma_{\varepsilon})$. Therefore, this completes the proof of Lemma~\ref{estofF}. 
\end{proof}


\section{Existence results}

In this section, we prove Proposition~\ref{MainProp} which is an abstract existence results of $\eqref{1.1}$ and then, we prove Theorem~\ref{Exist-thm} as an application of this proposition. First, we recall a supersolution method to construct a solution to $\eqref{1.1}$. 

\begin{prop}\label{equisuper}
  Let $0<T\le\infty$ and $f$ be a continuous and nondecreasing function such that $f(0) \ge 0$. Then, $\eqref{1.1}$ has a local-in-time nonnegative solution in $\RN\times(0, T)$ if and only if $\eqref{1.1}$ has a nonnegative supersolution $\overline{u} \in L^{\infty}((0, T), L^{1}_{\rm ul}(\RN)) \cap L^{\infty}_{\rm loc}((0, T), L^{\infty}(\RN))$. Moreover, if a nonnegative supersolution $\overline{u}$ exists, then the solution $u$ satisfies $0\le u(t) \le \overline{u}(t)$ for $0<t<T$. 
\end{prop}

\begin{proof}
A proof of Proposition \ref{equisuper} is well-known. See \cite[Proposition~2.4]{MS21} and \cite[Theorem~2.1]{RS13} for instance. 
\end{proof} 

We show the following proposition which is an abstract existence result of $\eqref{1.1}$ for a convex function. To this end, we prepare a function $H$. We now assume that a function $H \in C^{2}([0, \infty))$ satisfies the following: 
\begin{equation}\label{(H)}\tag{H}
  \begin{dcases}
    \lim_{\xi\to\infty} H(\xi) = \infty, &~ \\
    H(\xi)>0 &\text{for}~~\xi>0, \\
    H'(\xi)>0 &\text{for}~~\xi>0, \\
    H'(\xi)~\text{is nondecreasing} &\text{for large}~~\xi>0. 
  \end{dcases}
\end{equation}

\begin{prop}\label{MainProp}
    Let $N \geq 1$, $0 < \gamma < \min\{2, N\}$ and $u_0 \geq 0$. Let $H\in C^{2}([0, \infty))$ satisfy $\eqref{(H)}$. Suppose that $f \in C([0, \infty))$ and $f$ is a nonnegative and nondecreasing function in $(0,\infty)$. Suppose that there exist $\theta \in (0, 1]$ and $\zeta \geq 0$ such that one of the following holds{\rm :} 
    \begin{itemize}
      \item[$\mathrm{(i)}$] $H(u_0) \in L^{1}_{\mathrm{ul}}(\mathbb{R}^N)$ and 
      \begin{equation}\label{3.1}
        \lim_{\eta \to \infty} \Phi(\eta) \int_{\eta}^{\infty} \frac{\Psi(\tau) H'(\tau)}{H(\tau)^{1 + \frac{2-\gamma}{N}}} \left(1 - \left(\frac{H(\eta)}{H(\tau)}\right)^{\frac{2}{N}} \right)^{-\frac{\gamma}{2}} \,d\tau = 0, 
      \end{equation}

      \item[$\mathrm{(ii)}$] $H(u_0) \in \mathcal{L}^{1}_{\mathrm{ul}}(\mathbb{R}^N)$ and 
      \begin{equation}\label{3.2}
        \limsup_{\eta \to \infty} \Phi(\eta) \int_{\eta}^{\infty} \frac{\Psi(\tau) H'(\tau)}{H(\tau)^{1 + \frac{2-\gamma}{N}}} \left(1 - \left(\frac{H(\eta)}{H(\tau)}\right)^{\frac{2}{N}} \right)^{-\frac{\gamma}{2}} \,d\tau < \infty, 
      \end{equation}
    \end{itemize}
    where 
    \begin{equation*}
      \Phi(\xi) := \sup_{\zeta \leq \sigma \leq \xi} \frac{H'(\sigma)}{H(\sigma)^{1-\theta}} \quad \text{and} \quad \Psi(\xi) := \sup_{\zeta \leq \sigma \leq \xi} \frac{f(\sigma)}{H(\sigma)^{\theta}}. 
    \end{equation*}
    Then, there exists $T>0$ such that $\eqref{1.1}$ has a local-in-time nonnegative solution $u$ in $\RN \times (0,T)$. Moreover, there exists $C > 0$ depending on $N$ and $u_0$ such that 
    \begin{equation}\label{MainProp-est}
      \|H(u(t))\|_{L^{1}_{\mathrm{ul}}(\mathbb{R}^N)} \leq C \quad \text{for}~~0<t<T. 
    \end{equation}
\end{prop}

A proof of Proposition~\ref{MainProp} is based on the construction of a suitable supersolution. Specifically, we take $R_1>0$ such that $H$ is convex on $[R_1,\infty)$ and define a function $u_1$ as
\begin{equation}\label{u_1def}
  u_1(x):=\max\{u_0(x), 1, R_1, \zeta\}, \quad x\in\RN. 
\end{equation}
Then, we prove that the function $\overline{u}$ defined by
\begin{equation}\label{ssdef}
  \overline{u}(t):=H^{-1}(2e^{t\Delta}H(u_1)), \quad t>0. 
\end{equation}
 is a supersolution of a problem $\eqref{1.1}$. We first show the estimate of $\overline{u}$. 

\begin{lem}\label{3lem1}
  Assume the same conditions as in Proposition~$\ref{MainProp}$. Let $u_1$ and $\overline{u}$ be as in $\eqref{u_1def}$ and $\eqref{ssdef}$, respectively. Then, $\overline{u}$ belongs to $L^{\infty}((0, \infty), L^{1}_{\rm ul}(\RN))\cap L^{\infty}_{\rm loc}((0, \infty), L^{\infty}(\RN))$. Furthermore, there exists $C>0$ depending only on $N$, $\gamma$ and $\theta$ such that
  \begin{equation}\label{ssest1}
    e^{(t-s)\Delta}\left[|\cdot|^{-\gamma} H(\overline{u}(s))^{\theta}\right] \le C t^{\frac{\gamma}{2}} H(\overline{u}(t))^{\theta} (t-s)^{-\frac{\gamma}{2}} s^{-\frac{\gamma}{2}} \quad \text{for}~~0<s<t. 
  \end{equation}
\end{lem}

\begin{proof}[{\bf Proof of Lemma~\ref{3lem1}}]
  We first show $\overline{u}\in L^{\infty}((0, \infty), L^{1}_{\rm ul}(\RN))\cap L^{\infty}_{\rm loc}((0, \infty), L^{\infty}(\RN))$. Since $H(u_1) \in L^{1}_{\rm ul}(\RN)$, it follows from Propsition~\ref{LpLqul} (i) that 
  \begin{equation*}
    \|\overline{u}(t)\|_{\infty} \leq H^{-1}(2 C(1+t^{-\frac{N}{2}}) \|H(u_1)\|_{L^{1}_{\rm ul}(\RN)}) < \infty \quad \textrm{for}~~t>0. 
  \end{equation*}
  Hence, $\overline{u} \in L^{\infty}_{\rm loc}((0, \infty), L^{\infty}(\RN))$. Since $H$ is convex on $[R_1, \infty)$, $H^{-1}$ is concave on $[H(R_1), \infty)$ and we have
  \begin{equation*}
    H^{-1}(\xi) \le R_1 + (H^{-1})'(H(R_1)) (\xi - H(R_1)) \quad \text{for}~~\xi\ge H(R_1). 
  \end{equation*}
  Then, by Propsition~\ref{LpLqul} (i) and the fact that $2e^{t\Delta}H(u_1) \ge H(R_1)$ for $t>0$, we have
  \begin{align*}
    \|\overline{u}(t)\|_{L^{1}_{\rm ul}(\RN)} \le C \|e^{t\Delta} H(u_1) + C\|_{L^{1}_{\rm ul}(\RN)} & \le C \|e^{t\Delta} H(u_1)\|_{L^{1}_{\rm ul}(\RN)} + \| C \|_{L^{1}_{\rm ul}(\RN)}  \\
    &\le C \|H(u_1)\|_{L^{1}_{\rm ul}(\RN)} + C<\infty
  \end{align*}
  for $t>0$. Thus, $\overline{u} \in L^{\infty}((0, \infty), L^{1}_{\rm ul}(\RN))$. Therefore, we have $\overline{u}\in L^{\infty}((0, \infty), L^{1}_{\rm ul}(\RN))\cap L^{\infty}_{\rm loc}((0, \infty), L^{\infty}(\RN))$. 
  
  We next show $\eqref{ssest1}$. The proof is divided into three cases based on the range of $\theta$. 
  
  \medskip 
  \noindent
  \underline{{\bf Case~(i)}~$\theta\in(\gamma/N,1]$}\quad Since $\gamma/\theta<N$, we observe from Proposition~\ref{Jensen} (ii) and Lemma~\ref{keylem1} that 
  \begin{align*}
    e^{(t-s)\Delta}\left[|\cdot|^{-\gamma} H(\overline{u}(s))^{\theta} \right] &=e^{(t-s)\Delta}\left[|\cdot|^{-\frac{\gamma}{\theta}} H(\overline{u}(s)) \right]^{\theta} \\
    &\le \left[e^{(t-s)\Delta}|\cdot|^{-\frac{\gamma}{\theta}} H(\overline{u}(s))\right]^{\theta} \\
    &=2^{\theta} \left[e^{(t-s)\Delta}|\cdot|^{-\frac{\gamma}{\theta}} e^{s\Delta}H(u_1)\right]^{\theta} \\
    &\le C t^{\frac{\gamma}{2}} H(\overline{u}(t))^{\theta}(t-s)^{-\frac{\gamma}{2}}s^{-\frac{\gamma}{2}}
  \end{align*}
  for $0<s<t$. 
  
  \medskip 
  \noindent
  \underline{{\bf Case~(ii)}~$\theta\in(0,(N-\gamma)/N)$}\quad By H\"{o}lder's inequality, we have
  \begin{align*}
    &\left[e^{(t-s)\Delta} |\cdot|^{-\gamma} H(\overline{u}(s))^{\theta} \right](x)=\int_{\RN} G(x-y,t-s)|y|^{-\gamma} H(\overline{u}(y,s))^{\theta} \,dy \\
    &\qquad =\int_{\RN} \left(G(x-y,t-s)|y|^{-\frac{\gamma}{1-\theta}} \right)^{1-\theta} \left(G(x-y,t-s)H(\overline{u}(y,s)) \right)^{\theta} \,dy \\
    &\qquad \le \left(\int_{\RN} G(x-y,t-s)|y|^{-\frac{\gamma}{1-\theta}} \,dy \right)^{1-\theta} \left(\int_{\RN} G(x-y,t-s)H(\overline{u}(y,s)) \,dy \right)^{\theta} \\
    &\qquad =\left[e^{(t-s)\Delta} |\cdot|^{-\frac{\gamma}{1-\theta}} \right]^{1-\theta}(x) \left[e^{(t-s)\Delta} H(\overline{u}(s)) \right]^{\theta}(x)
  \end{align*}
  for $x\in\RN$ and $0<s<t$. Since $0<\gamma/(1-\theta)<N$, we see that
  \begin{align*}
    \left(e^{(t-s)\Delta} |\cdot|^{-\frac{\gamma}{1-\theta}} \right)^{1-\theta} &= \left((t-s)^{-\frac{\gamma}{2(1-\theta)}} [e^{\Delta} |\cdot|^{-\frac{\gamma}{1-\theta}}] \right)^{1-\theta} \\
    &\le C(t-s)^{-\frac{\gamma}{2}}
  \end{align*}
  for $0<s<t$. Moreover, by the definition of $\overline{u}$, we obtain
  \begin{align*}
    \left(e^{(t-s)\Delta} H(\overline{u}(s))\right)^{\theta} &=\left(2 e^{(t-s)\Delta}(e^{s\Delta}H(u_1)) \right)^{\theta} \\
    &=\left(2 e^{t\Delta}H(u_1) \right)^{\theta} \\
    &=H(\overline{u}(t))^{\theta}
  \end{align*}
  for $0<s<t$. Thus, we have
  \begin{align*}
    e^{(t-s)\Delta} \left[|\cdot|^{-\gamma} H(\overline{u}(s))^{\theta} \right] &\le C H(\overline{u}(t))^{\theta} (t-s)^{-\frac{\gamma}{2}} \\
    &\le C t^{\frac{\gamma}{2}} H(\overline{u}(t))^{\theta} (t-s)^{-\frac{\gamma}{2}} s^{-\frac{\gamma}{2}}
  \end{align*}
  for $0<s<t$. 
    
  \medskip 
  \noindent
  \underline{{\bf Case~(iii)}~$\theta\in[(N-\gamma)/N,\gamma/N]$}\quad We take $\theta_1$ and $\theta_2$ satisfying
  \begin{equation*}
    0<\theta_1<\frac{N-\gamma}{N} \quad \text{and} \quad \frac{\gamma}{N}<\theta_2\le1, 
  \end{equation*}
  respectively. Since $\theta\in[(N-\gamma)/N, \gamma/N]$, there exists $\alpha\in(0,1)$ such that
  \begin{equation*}
    \theta=(1-\alpha)\theta_1+\alpha\theta_2. 
  \end{equation*}
  Thus, it follows from H\"{o}lder's inequality that
  \begin{align*}
    &\left[e^{(t-s)\Delta} |\cdot|^{-\gamma} H(\overline{u}(s))^{\theta} \right](x)=\int_{\RN} G(x-y,t-s)|y|^{-\gamma} H(\overline{u}(y,s))^{(1-\alpha)\theta_1} H(\overline{u}(y,s))^{\alpha\theta_2} \,dy \\
    &\quad= \int_{\RN} \left[G(x-y,t-s)|y|^{-\gamma} H(\overline{u}(y,s))^{\theta_1} \right]^{(1-\alpha)} \left[G(x-y,t-s)|y|^{-\gamma} H(\overline{u}(y,s))^{\theta_2} \right]^{\alpha} \,dy \\
    &\quad \le \left(\int_{\RN} G(x-y,t-s)|y|^{-\gamma} H(\overline{u}(y,s))^{\theta_1}\,dy \right)^{(1-\alpha)} \left(\int_{\RN} G(x-y,t-s)|y|^{-\gamma} H(\overline{u}(y,s))^{\theta_2}\,dy \right)^{\alpha}
  \end{align*}
  for $x\in\RN$ and $0<s<t$. Since $\theta_1\in(0,(N-\gamma)/N)$ and $\theta_2\in(\gamma/N,1]$, by the argument in the cases (i) and (ii), we have
  \begin{align*}
    &e^{(t-s)\Delta} \left[|\cdot|^{-\gamma} H(\overline{u}(s))^{\theta} \right] \\
    &\le C\left(t^{\frac{\gamma}{2}} H(\overline{u}(t))^{\theta_1} (t-s)^{-\frac{\gamma}{2}} s^{-\frac{\gamma}{2}} \right)^{(1-\alpha)} \left(t^{\frac{\gamma}{2}} H(\overline{u}(t))^{\theta_2} (t-s)^{-\frac{\gamma}{2}} s^{-\frac{\gamma}{2}} \right)^{\alpha} \\
    &=C t^{\frac{\gamma}{2}} H(\overline{u}(t))^{\theta} (t-s)^{-\frac{\gamma}{2}} s^{-\frac{\gamma}{2}}
  \end{align*}
  for $0<s<t$. Therefore, we have $\eqref{ssest1}$ and hence, this completes the proof of Lemma~\ref{3lem1}. 
\end{proof}

We move to prove Proposition~\ref{MainProp}. 

\begin{proof}[{\bf Proof of Proposition~\ref{MainProp}}]
  Let $u_1$ and $\overline{u}$ be as in $\eqref{u_1def}$ and $\eqref{ssdef}$, respectively. We see that $H(u_1)\in L^{1}_{\rm ul}(\RN)$ in the case (i). By Lemma~\ref{Lulmax}, we have that $H(u_1) \in \cL^{1}_{\rm ul}(\RN)$ in the case (ii). We note that since $u_1\ge R_1$, we can use the convexity of $H$ on $[R_1,\infty)$. Then, we have
  \begin{equation*}
    H(u_1) \ge H(R_1) + H'(R_1)(u_1-R_1). 
  \end{equation*}
  Thus, we obtain $u_1 \in L^{1}_{\rm ul}(\RN)$, and hence $u_0 \in L^{1}_{\rm ul}(\RN)$. 
  
  We now show that $\overline{u}$ is a supersolution of $\eqref{1.1}$. By Lemma~\ref{3lem1}, we have $\overline{u}\in L^{\infty}((0, \infty), L^{1}_{\rm ul}(\RN))\cap L^{\infty}_{\rm loc}((0, \infty), L^{\infty}(\RN))$. Thus, we show that there exists $T>0$ such that
  \begin{equation}\label{ssineq}
    \overline{u}(t) \ge e^{t\Delta}u_0 + \int_{0}^{t} e^{(t-s)\Delta} |\cdot|^{-\gamma} f(\overline{u}(s)) \,ds
  \end{equation}
  for $t \in (0, T)$. Fix $t>0$. Since $[e^{t\Delta}u_1](\cdot)\ge R_1$, we observe from Proposition~\ref{Jensen} (i) and mean value theorem that
  \begin{equation}\label{3.8}
    \begin{split}
      \overline{u}(t) - e^{t\Delta} u_0 &\ge \overline{u}(t) - e^{t\Delta} u_1 \\
      &\ge H^{-1}(2 e^{t\Delta} H(u_1)) - H^{-1}(e^{t\Delta} H(u_1)) \\
      &=(H^{-1})'((1+c) e^{t\Delta} H(u_1)) ~e^{t\Delta} H(u_1) \\
      &=\frac{H(\overline{u}(t))}{2 H'(H^{-1}((1+c)e^{t\Delta}H(u_1)))}
    \end{split}
  \end{equation}
  for some $c=c(x,t)\in[0,1]$. Since $H^{-1}$ is nondecreasing, we have
  \begin{equation}\label{3.9}
    \begin{split}
      H^{-1}((1+c) e^{t\Delta} H(u_1)) \le H^{-1}(2e^{t\Delta}H(u_1))=\overline{u}(t) 
    \end{split}
  \end{equation}
  Since $H'$ is nondecreasing, we observe from $\eqref{3.8}$ and $\eqref{3.9}$ that
  \begin{equation}\label{3.10}
    \overline{u}(t) - e^{t\Delta} u_0 \ge \frac{1}{2} \frac{H(\overline{u}(t))}{H'(\overline{u}(t))}. 
  \end{equation}
  On the other hand, it follows from Lemma~\ref{3lem1} that
  \begin{equation}\label{3.11}
    \begin{split}
      \int_0^t e^{(t-s)\Delta}|\cdot|^{-\gamma} f(\overline{u}(s))\,ds &\le \int_0^t \left\|\frac{f(\overline{u}(s))}{H(\overline{u}(s))^{\theta}} \right\|_{\infty} \left[e^{(t-s)\Delta}|\cdot|^{-\gamma} H(\overline{u}(s))^{\theta} \right]\,ds \\
      &\le C t^{\frac{\gamma}{2}} H(\overline{u}(t))^{\theta} \int_0^t \left\|\frac{f(\overline{u}(s))}{H(\overline{u}(s))^{\theta}} \right\|_{\infty} (t-s)^{-\frac{\gamma}{2}} s^{-\frac{\gamma}{2}}\,ds. 
    \end{split}
  \end{equation}
  Therefore, by $\eqref{3.10}$ and $\eqref{3.11}$, it suffices to show the following to prove $\eqref{ssineq}$: 
  \begin{equation}\label{3.12}
    \limsup_{t\to0} t^{\frac{\gamma}{2}} \frac{H'(\overline{u}(t))}{H(\overline{u}(t))^{1-\theta}} \int_{0}^{t} \left\|\frac{f(\overline{u}(s))}{H(\overline{u}(s))^{\theta}} \right\|_{\infty} (t-s)^{-\frac{\gamma}{2}} s^{-\frac{\gamma}{2}} \,ds =0. 
  \end{equation}
  We first consider the case (i). By the definitions of $\Phi$ and $\Psi$, we have 
  \begin{equation}\label{3.13}
    \begin{split}
      t^{\frac{\gamma}{2}} \frac{H'(\overline{u}(t))}{H(\overline{u}(t))^{1-\theta}} &\int_{0}^{t} \left\|\frac{f(\overline{u}(s))}{H(\overline{u}(s))^{\theta}} \right\|_{\infty} (t-s)^{-\frac{\gamma}{2}} s^{-\frac{\gamma}{2}} \,ds \\
      &\le t^{\frac{\gamma}{2}} \Phi(\|\overline{u}(t)\|_{\infty}) \int_{0}^{t} \Psi(\|\overline{u}(s)\|_{\infty}) (t-s)^{-\frac{\gamma}{2}} s^{-\frac{\gamma}{2}} \,ds. 
    \end{split}
  \end{equation}
  Now, it follows from Proposition~\ref{LpLqul}~(i) that
  \begin{equation*}
    \|\overline{u}(t)\|_{\infty} \le H^{-1}\left(2 C t^{-\frac{N}{2}} \|H(u_1)\|_{L^{1}_{\rm ul}(\RN)}\right) =:\eta(t)
  \end{equation*}
  for $0<t<1$. Then, we have 
  \begin{equation}\label{3.14}
    \begin{split}
      t^{\frac{\gamma}{2}} \Phi&(\|\overline{u}(t)\|_{\infty}) \int_{0}^{t} \Psi(\|\overline{u}(s)\|_{\infty}) (t-s)^{-\frac{\gamma}{2}} s^{-\frac{\gamma}{2}} \,ds \\
      &\le t^{\frac{\gamma}{2}} \Phi(\eta(t)) \int_{0}^{t} \Psi\left(\eta(s) \right) (t-s)^{-\frac{\gamma}{2}} s^{-\frac{\gamma}{2}} \,ds \\
      &=C \|H(u_1)\|_{L^{1}_{\rm ul}(\RN)}^{\frac{2-\gamma}{N}} \Phi(\eta(t)) \int_{\eta(t)}^{\infty} \frac{\Psi(\tau) H'(\tau)}{H(\tau)^{1+\frac{2-\gamma}{N}}} \left(1-\left(\frac{H(\eta(t))}{H(\tau)} \right)^{\frac{2}{N}} \right)^{-\frac{\gamma}{2}} \,d\tau
    \end{split}
  \end{equation}
  for $0<t<1$. Here we used a change of variables $\tau=\eta(s)$ in the last equality. 
  By $\eqref{3.13}$ and $\eqref{3.14}$, we have
  \begin{align*}
    t^{\frac{\gamma}{2}} \frac{H'(\overline{u}(t))}{H(\overline{u}(t))^{1-\theta}} &\int_{0}^{t} \left\|\frac{f(\overline{u}(s))}{H(\overline{u}(s))^{\theta}} \right\|_{\infty} (t-s)^{-\frac{\gamma}{2}} s^{-\frac{\gamma}{2}} \,ds \\
    &\le C \|H(u_1)\|_{L^{1}_{\rm ul}(\RN)}^{\frac{2-\gamma}{N}} \Phi(\eta(t)) \int_{\eta(t)}^{\infty} \frac{\Psi(\tau) H'(\tau)}{H(\tau)^{1+\frac{2-\gamma}{N}}} \left(1-\left(\frac{H(\eta(t))}{H(\tau)} \right)^{\frac{2}{N}} \right)^{-\frac{\gamma}{2}} \,d\tau
  \end{align*}
  for $0<t<1$. Thus, $\eqref{3.12}$ follows from assumption $\eqref{3.1}$. 

  We next consider the case (ii). Since $H(u_1)\in\cL^{1}_{\rm ul}(\RN)$, it follows from Proposition~\ref{LpLqul} (ii) that for any $\varepsilon>0$, there exists $t_{\varepsilon}>0$ such that 
  \begin{equation*}
    \|\overline{u}(t)\|_{\infty} \leq H^{-1}(2 \varepsilon t^{-\frac{N}{2}} \|H(u_1)\|_{L^{1}_{\rm ul}(\RN)}) =:\eta_{\varepsilon}(t) \quad \text{for}~~0<t<t_{\varepsilon}. 
  \end{equation*}
  By the same argument as in the case (i), we have
  \begin{align*}
    t^{\frac{\gamma}{2}} \frac{H'(\overline{u}(t))}{H(\overline{u}(t))^{1-\theta}} &\int_{0}^{t} \left\|\frac{f(\overline{u}(s))}{H(\overline{u}(s))^{\theta}} \right\|_{\infty} (t-s)^{-\frac{\gamma}{2}} s^{-\frac{\gamma}{2}} \,ds \\
    &\le C \varepsilon^{\frac{2-\gamma}{N}} \|H(u_1)\|_{L^{1}_{\rm ul}(\RN)}^{\frac{2-\gamma}{N}} \Phi(\eta_{\varepsilon}(t)) \int_{\eta_{\varepsilon}(t)}^{\infty} \frac{\Psi(\tau) H'(\tau)}{H(\tau)^{1+\frac{2-\gamma}{N}}} \left(1-\left(\frac{H(\eta_{\varepsilon}(t))}{H(\tau)} \right)^{\frac{2}{N}} \right)^{-\frac{\gamma}{2}} \,d\tau
  \end{align*}
  for $0<t<t_{\varepsilon}$. By $\eqref{3.2}$, we have
  \begin{equation*}
    \limsup_{t\to0} t^{\frac{\gamma}{2}} \frac{H'(\overline{u}(t))}{H(\overline{u}(t))^{1-\theta}} \int_{0}^{t} \left\|\frac{f(\overline{u}(s))}{H(\overline{u}(s))^{\theta}} \right\|_{\infty} (t-s)^{-\frac{\gamma}{2}} s^{-\frac{\gamma}{2}} \,ds \le C \varepsilon^{\frac{2-\gamma}{N}} \|H(u_1)\|_{L^{1}_{\rm ul}(\RN)}^{\frac{2-\gamma}{N}}. 
  \end{equation*}
  Since $\varepsilon>0$ is arbitrary, we obtain $\eqref{3.12}$. 
  
  By $\eqref{ssineq}$, there exists $T>0$ such that $\overline{u}$ is a supersolution of $\eqref{1.1}$ in $\RN \times (0, T)$ satisfying $\overline{u}\in L^{\infty}((0, \infty), L^{1}_{\rm ul}(\RN))\cap L^{\infty}_{\rm loc}((0, \infty), L^{\infty}(\RN))$. By Proposition~\ref{equisuper}, we see that $\eqref{1.1}$ has a local-in-time nonnegative solution $u$ in $\RN \times (0, T)$ and that $0 \le u(t) \le \overline{u}(t)$ for $0 < t < T$. Since $H(u_1) \in L^{1}_{\rm ul}(\RN)$, we have 
  \begin{equation*}
    \|H(u(t))\|_{L^{1}_{\rm ul}(\RN)} \le \|H(\overline{u}(t))\|_{L^{1}_{\rm ul}(\RN)} = C \|e^{t\Delta} H(u_1)\|_{L^{1}_{\rm ul}(\RN)} \le C \|H(u_1)\|_{L^{1}_{\rm ul}(\RN)} <\infty 
  \end{equation*}
  for $0<t<T$. Therefore, we obtain $\eqref{MainProp-est}$. This completes the proof of Proposition~\ref{MainProp}. 
\end{proof}

We next relate $H$ and nonlinear term $f$ with $H(\xi):=F(\xi)^{-r}$ and show Theorem~\ref{Exist-thm} by using Proposition~\ref{MainProp}. 

\begin{proof}[\bf Proof of Theorem~\ref{Exist-thm}]
  In all cases, since $f\in S_{q}$, we see that $f \in C([0, \infty))$ and $f$ is nonnegative and nondecreasing in $(0, \infty)$. In the following, we will show that it is possible to apply Proposition~\ref{MainProp} with $H(\xi) := F(\xi)^{-r}$. 
  
  We first show that $H$ satisfies $\eqref{(H)}$. Since $f \in S_{q}$, we have $H \in C^{2}([0, \infty))$, $\lim_{\xi\to\infty} H(\xi) = \infty$, $H(\xi)>0$ and
  \begin{equation*}
    H'(\xi)=\frac{r}{f(\xi) F(\xi)^{r+1}}>0, \qquad H''(\xi) = \frac{r((r+1) - f'(\xi) F(\xi))}{f(\xi)^{2} F(\xi)^{r+2}}
  \end{equation*}
  for $\xi>0$. It remains to show that
  \begin{equation}\label{3.15}
    H''(\xi) \ge 0 \quad \text{for large}~~\xi>0. 
  \end{equation}
  In the cases (i) and (iii), since $q<1+r$, we have $(r+1) - f'(\xi) F(\xi)\to r+1-q>0$ as $\xi \to \infty$, and hence $\eqref{3.15}$ holds. In the case (ii), by $\eqref{sub-assum}$, there exists $R_2>0$ such that $q=r+1\ge f'(\xi)F(\xi)$ for $\xi>R_2$ and hence $\eqref{3.15}$ holds. Thus, $H$ satisfies $\eqref{(H)}$. 

  We consider the case (i). First, we show that there exists $\theta\in(0,1]$ such that
  \begin{equation}\label{*1}
    \frac{d}{d\xi} \left(\frac{f(\xi)}{H(\xi)^{\theta}} \right) \ge 0 \quad \text{for large}~~\xi>0
  \end{equation}
  and 
  \begin{equation}\label{*2}
    \frac{d}{d\xi} \left(\frac{H'(\xi)}{H(\xi)^{1-\theta}} \right) \ge 0 \quad \text{for large}~~\xi>0. 
  \end{equation}
  By the direct calculation, we have
  \begin{equation*}
    \frac{d}{d\xi} \left(\frac{f(\xi)}{H(\xi)^{\theta}} \right) = (f'(\xi)F(\xi)-r\theta) F(\xi)^{r\theta - 1}
  \end{equation*}
  and
  \begin{equation*}
    \frac{d}{d\xi} \left(\frac{H'(\xi)}{H(\xi)^{1-\theta}} \right) = (1-\theta) \frac{H''(\xi)}{H(\xi)^{1-\theta}} \left(\frac{1}{1-\theta}-\frac{r}{(r+1)-f'(\xi)F(\xi)} \right). 
  \end{equation*}
  Moreover, since $f\in S_{q}$, we see that
  \begin{equation*}
    \lim_{\xi\to\infty} (f'(\xi)F(\xi)-r\theta)=q-r\theta \quad \text{and} \quad \lim_{\xi\to\infty} \frac{r}{(r+1)-f'(\xi)F(\xi)}=\frac{r}{r+1-q}. 
  \end{equation*}
  Thus, if we take $\theta$ which satisfies
  \begin{equation*}
    \theta\in (0,1)\quad \text{and}\quad \frac{q-1}{r}<\theta<\min\left\{1,\frac{q}{r}\right\}, 
  \end{equation*}
  then we obtain $\eqref{*1}$ and $\eqref{*2}$. 

  Next, we check that $H$ satisfies $\eqref{3.1}$. By $\eqref{*1}$ and $\eqref{*2}$, we can take $\zeta>0$ such that 
  \begin{equation*}
    \Phi(\xi)=\frac{H'(\xi)}{H(\xi)^{1-\theta}},~~\Psi(\xi)=\frac{f(\xi)}{H(\xi)^{\theta}} 
  \end{equation*}
  for $\xi>\zeta$ and $\Phi$, $\Psi$ are nondecreasing on $(\zeta, \infty)$. Then, we have
  \begin{equation}\label{3.18}
    \begin{split}
      \Phi(\eta) \int_{\eta}^{\infty} \frac{\Psi(\tau) H'(\tau)}{H(\tau)^{1 + \frac{2-\gamma}{N}}} &\left(1 - \left(\frac{H(\eta)}{H(\tau)}\right)^{\frac{2}{N}} \right)^{-\frac{\gamma}{2}} \,d\tau \\
      &\le \int_{\eta}^{\infty} \frac{f(\tau) (H'(\tau))^{2}}{H(\tau)^{2 + \frac{2-\gamma}{N}}} \left(1 - \left(\frac{H(\eta)}{H(\tau)}\right)^{\frac{2}{N}} \right)^{-\frac{\gamma}{2}} \,d\tau \\
      &=r^{2} \int_{\eta}^{\infty} \frac{F(\tau)^{\frac{(2-\gamma)r}{N} - 2}}{f(\tau)} \left(1- \left(\frac{F(\eta)}{F(\tau)} \right)^{-\frac{2r}{N}} \right)^{-\frac{\gamma}{2}} \,d\tau 
    \end{split}
  \end{equation}
  for $\eta>\zeta$. By a change of variables $(F(\eta)/F(\tau))^{-2r/N}=\sigma$, we have
  \begin{equation}\label{3.19}
    \begin{split}
      r^{2} \int_{\eta}^{\infty} \frac{F(\tau)^{\frac{(2-\gamma)r}{N} - 2}}{f(\tau)} &\left(1- \left(\frac{F(\eta)}{F(\tau)} \right)^{-\frac{2r}{N}} \right)^{-\frac{\gamma}{2}} \,d\tau \\
      &=\frac{rN}{2} F(\eta)^{\frac{(2-\gamma)r}{N} - 1} \int_{0}^{1} \sigma^{\frac{N}{2r} \left(\frac{(2-\gamma)r}{N} - 1 \right) - 1} (1-\sigma)^{-\frac{\gamma}{2}} \,d\sigma
    \end{split}
  \end{equation}
  for $\eta>\zeta$. Since 
  \begin{equation*}
    \frac{(2-\gamma)r}{N}-1>0,~~\frac{N}{2r} \left(\frac{(2-\gamma)r}{N}-1 \right)-1>-1~~\text{and}~~ -\frac{\gamma}{2}>-1, 
  \end{equation*}
  we have
  \begin{equation*}
    \int_{0}^{1} \sigma^{\frac{N}{2r} \left(\frac{(2-\gamma)r}{N} - 1 \right) - 1} (1-\sigma)^{-\frac{\gamma}{2}} \,d\sigma < \infty
  \end{equation*}
  and hence
  \begin{equation}\label{3.20}
    \frac{rN}{2} F(\eta)^{\frac{(2-\gamma)r}{N} - 1} \int_{0}^{1} \sigma^{\frac{N}{2r} \left(\frac{(2-\gamma)r}{N} - 1 \right)-1} (1-\sigma)^{-\frac{\gamma}{2}} \,d\sigma \to 0 \quad \text{as}~~\eta\to\infty. 
  \end{equation}
  By $\eqref{3.18}$, $\eqref{3.19}$ and $\eqref{3.20}$, we see that $H$ satisfies all assumptions of Proposition~\ref{MainProp} (i). 
  
  We next consider the case (ii). Since $q=1+r$ and $f'(\xi)F(\xi)-r\to1>0$ as $\xi\to\infty$, there exists $R_3>0$ such that
  \begin{equation*}
    \frac{d}{d\xi} \left(\frac{f(\xi)}{H(\xi)} \right) = F(\xi)^{r-1} (f'(\xi) F(\xi)-r) \ge 0 \quad \text{for}~~\xi>R_3, 
  \end{equation*}
  Furthermore, by $\eqref{sub-assum}$, there exists $R_4>0$ such that
  \begin{equation*}
    \frac{d}{d\xi} H'(\xi) = H''(\xi) = \frac{r((r+1) - f'(\xi) F(\xi))}{f(\xi)^{2} F(\xi)^{r+2}} \ge 0 \quad \text{for}~~\xi>R_4. 
  \end{equation*}
  Thus, $\eqref{*1}$ and $\eqref{*2}$ with $\theta=1$ hold. Therefore, by the same argument as in the case (i), $H$ satisfies all assumptions of Proposition~\ref{MainProp} (i). Therefore, in the cases (i) and (ii), we see that $H$ satisfies all assumptions of Proposition~\ref{MainProp} (i) and thus, there exists $T>0$ such that $\eqref{1.1}$ admits a local-in-time nonnegative solution $u$ in $\RN \times(0,T)$. 

  Finally, we consider the case (iii). By the same argument as in the case (i), $\eqref{*1}$ and $\eqref{*2}$ also hold. We check that $H$ satisfies $\eqref{3.2}$. By $\eqref{*1}$ and $\eqref{*2}$, there exists $R_5>0$ such that
  \begin{equation*}
    \Phi(\xi)=\frac{H'(\xi)}{H(\xi)^{1-\theta}},\quad \Psi(\xi)=\frac{f(\xi)}{H(\xi)^{\theta}} 
  \end{equation*}
  for $\xi>R_5$ and $\Phi$, $\Psi$ are nondecreasing on $(R_5,\infty)$. Since $r=N/(2-\gamma)$, we have 
  \begin{equation*}
    \begin{split}
      \Phi(\eta) &\int_{\eta}^{\infty} \frac{\Psi(\tau) H'(\tau)}{H(\tau)^{1+\frac{2-\gamma}{N}}} \left(1-\left(\frac{H(\eta)}{H(\tau)} \right)^{\frac{2}{N}} \right)^{-\frac{\gamma}{2}} \,d\tau \\
      &=\frac{H'(\eta)}{H(\eta)^{1-\theta}} \int_{\eta}^{\infty} \frac{f(\tau)H'(\tau)}{H(\tau)^{1+\frac{2-\gamma}{N}+\theta}} \left(1-\left(\frac{H(\eta)}{H(\tau)} \right)^{\frac{2}{N}} \right)^{-\frac{\gamma}{2}} \,d\tau \\
      &=\frac{r^{2}}{f(\eta)F(\eta)^{r\theta+1}} \int_{\eta}^{\infty} F(\tau)^{r\theta + \frac{2-\gamma}{N} r - 1} \left(1-\left(\frac{F(\tau)}{F(\eta)} \right)^{\frac{2r}{N}} \right)^{-\frac{\gamma}{2}} \,d\tau \\
      &= \left(\frac{N}{2-\gamma} \right)^{2} \frac{1}{f(\eta)F(\eta)^{\frac{N}{2-\gamma}\theta+1}} \int_{\eta}^{\infty} F(\tau)^{\frac{N}{2-\gamma}\theta} \left(1-\left(\frac{F(\tau)}{F(\eta)} \right)^{\frac{2}{2-\gamma}} \right)^{-\frac{\gamma}{2}} \,d\tau \\
      &=\frac{N^{2}}{2(2-\gamma)} \frac{1}{f(\eta)} \int_{0}^{1} \sigma^{\frac{N}{2}\left(\theta + \frac{2-\gamma}{N} \right)-1} (1-\sigma)^{-\frac{\gamma}{2}} f(F^{-1}(F(\eta) \sigma^{\frac{2-\gamma}{2}})) \,d\sigma
    \end{split}
  \end{equation*}
  for $\eta>R_5$. Here, we used the change of variables $(F(\tau)/F(\eta))^{2/(2-\gamma)} = \sigma$ in the last equality. This together with Proposition~\ref{keyprop} implies that $H$ satisfies all assumptions of Proposition~\ref{MainProp} (ii). Therefore, by Proposition~\ref{MainProp} (ii), we see that there exists $T>0$ such that $\eqref{1.1}$ admits a local-in-time nonnegative solution $u$ in $\RN\times(0,T)$. Furthermore, in all cases, the estimate $\eqref{Exist-thm-est}$ follows from $\eqref{MainProp-est}$. This completes the proof of Theorem~\ref{Exist-thm}. 
\end{proof}


\section{Nonexistence result}
In this section, we prove Theorem~\ref{Nonexist-thm}. We first give the necessary condition of the existence of a local-in-time nonnegative solution to $\eqref{1.1}$. 

\begin{lem}\label{Nonexist-lem}
  Suppose that $f$ satisfies the same assumptions as in Theorem~$\ref{Nonexist-thm}$. Let $u_0\in L^{1}_{\rm ul}(\RN)$ and assume that there exists a local-in-time nonnegative solution $u$ to problem $\eqref{1.1}$ in $\RN\times(0,T)$ with initial function $u_0$ for some $T>0$. Then, there exists $C>0$ depending on $\delta$, $N$ and $\gamma$ such that 
  \begin{equation*}
    [e^{t\Delta} u_0](0) \le F^{-1}(C t^{1-\frac{\gamma}{2}})
  \end{equation*}
  for $t\in(0,T)$. 
\end{lem}

\begin{proof}
  Fix $t\in(0,T)$ and take $\tau\in(0,t)$. Since $u$ is a solution to $\eqref{1.1}$, it follows from Fubini's theorem that
  \begin{equation}\label{Nonexist-pr-eq1}
    [e^{(t-\tau)\Delta} u(\tau)](x) = [e^{t\Delta} u_0](x) + \int_{0}^{\tau} [e^{(t-s)\Delta} |\cdot|^{-\gamma} f(u(\cdot, s))](x) \,ds
  \end{equation}
  for $x \in \RN$. Let $\delta>0$ be small enough such that $\eqref{Nonexist-assum}$ holds. Then, we have
  \begin{equation*}
    \frac{d^{2}}{du^{2}} (f(u)^{1-\delta}) = (1-\delta) f(u)^{-\delta-1} \{f''(u) f(u) - \delta (f'(u))^{2}\} \ge 0 \quad \text{for}~~u\ge 0. 
  \end{equation*}
  Thus, $f^{1-\delta}$ is convex on $[0, \infty)$. Thus, it follows from Proposition~\ref{Jensen} and the reverse H\"{o}lder's inequality that 
  \begin{align*}
    &[e^{(t-s)\Delta} |\cdot|^{-\gamma} f(u(\cdot, s))](x) \\
    &= \int_{\RN} G(x-y, t-s) |y|^{-\gamma} f(u(y, s)) \,dy \\
    &=\int_{\RN} (G(x-y, t-s) |y|^{\frac{1-\delta}{\delta}\gamma})^{-\frac{\delta}{1-\delta}} (G(x-y, t-s) f(u(y,s))^{1-\delta})^{\frac{1}{1-\delta}} \,dy \\
    &\ge\left(\int_{\RN} G(x-y, t-s) |y|^{\frac{1-\delta}{\delta}\gamma} \,dy \right)^{-\frac{\delta}{1-\delta}} \left(\int_{\RN} G(x-y, t-s) f(u(y, s))^{1-\delta} \,dy \right)^{\frac{1}{1-\delta}} \\
    &=\left([e^{(t-s)\Delta} |\cdot|^{\frac{1-\delta}{\delta}\gamma}](x) \right)^{-\frac{\delta}{1-\delta}} \left([e^{(t-s)\Delta} f(u(\cdot, s))^{1-\delta}](x) \right)^{\frac{1}{1-\delta}} \\
    &\ge \left([e^{(t-s)\Delta} |\cdot|^{\frac{1-\delta}{\delta}\gamma}](x) \right)^{-\frac{\delta}{1-\delta}} f([e^{(t-s)\Delta} u(\cdot, s)](x))
  \end{align*}
  for $x\in\RN$ and $s\in(0, \tau)$. We now define $H(x, \tau)$ as the right hand side of $\eqref{Nonexist-pr-eq1}$, i.e., 
  \begin{equation*}
    H(x, \tau) := [e^{t\Delta} u_0](x) + \int_{0}^{\tau} [e^{(t-s)\Delta} |\cdot|^{-\gamma} f(u(\cdot, s))](x) \,ds. 
  \end{equation*}
  Then, we note that $[e^{(t-\tau)\Delta}u(\tau)](x)=H(x,\tau)$ and for fixed $x\in\RN$, $H(x, \tau)$ is absolutely continuous with respect to $\tau$ on $(0,t)$. Thus, we see that $H$ is differentiable a.e.~in $(0, t)$ and hence, we have
  \begin{align*}
    \frac{\partial}{\partial \tau} H(x, \tau) &= [e^{(t-\tau)\Delta} |\cdot|^{-\gamma} f(u(\cdot, \tau))](x) \\
    &\ge \left([e^{(t-\tau)\Delta} |\cdot|^{\frac{1-\delta}{\delta}\gamma}](x) \right)^{-\frac{\delta}{1-\delta}} f([e^{(t-\tau)\Delta} u(\cdot, \tau)](x)) \\
    &=\left([e^{(t-t)\Delta} |\cdot|^{\frac{1-\delta}{\delta}\gamma}](x) \right)^{-\frac{\delta}{1-\delta}} f(H(x, \tau))
  \end{align*}
  for a.a. $\tau\in(0,t)$. Therefore, for $x \in \RN$ and a.a. $\tau\in(0,t)$, we obtain
  \begin{equation*}
    -\frac{\partial}{\partial \tau} F(H(x, \tau)) = \frac{\frac{\partial}{\partial \tau} H(x, \tau)}{f(H(x, \tau))} \ge \left([e^{(t-\tau)\Delta} |\cdot|^{\frac{1-\delta}{\delta}\gamma}](x) \right)^{-\frac{\delta}{1-\delta}}. 
  \end{equation*}
  Integrating the both sides of the above inequality over $(0, t)$, we have 
  \begin{equation*}
    -F(H(x, t)) + F(H(x, 0)) \ge \int_{0}^{t} \left([e^{(t-\tau)\Delta} |\cdot|^{\frac{1-\delta}{\delta}\gamma}](x) \right)^{-\frac{\delta}{1-\delta}} \,d\tau \quad \text{for}~~x\in\RN. 
  \end{equation*}
  Since $F(H(x, t)) > 0$ and $F(H(x, 0)) = F([e^{t\Delta} u_0](x))$, we obtain
  \begin{equation*}
    F([e^{t\Delta} u_0](x)) \ge \int_{0}^{t} \left([e^{(t-\tau)\Delta} |\cdot|^{\frac{1-\delta}{\delta}\gamma}](x) \right)^{-\frac{\delta}{1-\delta}} \,d\tau \quad \text{for}~~x \in \RN. 
  \end{equation*}
  By the simple calculation, we see that
  \begin{align*}
    \int_{0}^{t} ([e^{(t-\tau)\Delta} |\cdot|^{\frac{1-\delta}{\delta}\gamma}](0))^{-\frac{\delta}{1-\delta}} \,d\tau &= \left\{(4\pi)^{-\frac{N}{2}} \int_{\RN} e^{-\frac{|z|^{2}}{4}} |z|^{\frac{1-\delta}{\delta}\gamma} \,dz \right\}^{-\frac{\delta}{1-\delta}} \int_{0}^{t} (t-\tau)^{-\frac{\gamma}{2}} \,d\tau \\
    &= \left\{(4\pi)^{-\frac{N}{2}} \int_{\RN} e^{-\frac{|z|^{2}}{4}} |z|^{\frac{1-\delta}{\delta}\gamma} \,dz \right\}^{-\frac{\delta}{1-\delta}} \int_{0}^{1} (1-\tau)^{-\frac{\gamma}{2}} \,d\tau \cdot t^{1-\frac{\gamma}{2}} \\
    &=C t^{1-\frac{\gamma}{2}}. 
  \end{align*}
  Therefore, we have 
  \begin{equation*}
    F([e^{t\Delta} u_0](0)) \ge C t^{1-\frac{\gamma}{2}}. 
  \end{equation*}
  This completes the proof of Lemma~\ref{Nonexist-lem}. 
\end{proof}

We are now ready to prove Theorem~\ref{Nonexist-thm}. 

\begin{proof}[\bf Proof of Theorem~\ref{Nonexist-thm}]
  Since $r$ satisfies $0<r<N/(2-\gamma)$, we can choose $A$ such that
  \begin{equation*}
    2-\gamma < A < \frac{N}{r}. 
  \end{equation*}
  We define 
  \begin{equation}\label{Nonexist-u_0def}
    u_0(x) := 
    \begin{dcases}
      F^{-1}(|x|^{A}) & \text{if}~~F(0) = \infty, \\
      F^{-1}(\min\{|x|^{A}, F(0)\}) & \text{if}~~F(0)<\infty. 
    \end{dcases}
  \end{equation}
  Then, we easily see that $F(u_0)^{-r} \in L^{1}_{\rm ul}(\RN)$. We show that $u_0 \in L^{1}_{\rm ul}(\RN)$. Let $\varepsilon_0>0$ satisfy
  \begin{equation*}
    0<\varepsilon_0 < \frac{N}{A}-(q-1). 
  \end{equation*}
  Then, by Lemma~\ref{estofF}, we see that there exists $\delta_1>0$ such that
  \begin{equation*}
    F^{-1}(\sigma) \le C \sigma^{-(q-1+\varepsilon_0)} 
  \end{equation*}
  for $\sigma \in (0, \delta_1)$. Let $R>0$ be small. Then, we have
  \begin{equation*}
    \int_{B(0, R)} F^{-1}\left(|x|^{A} \right) \,dx \le C \int_{0}^{R} \tau^{-A(q-1+\varepsilon_0) + (N-1)} \,d\tau=C\int_{0}^{R} \tau^{A\left(\frac{N}{A} - (q-1) -\varepsilon_0 \right)-1} \,d\tau <\infty. 
  \end{equation*}
  Therefore, we obtain $u_0 \in L^{1}_{\rm ul}(\RN)$. 
  
  In the following, we show that problem $\eqref{1.1}$ admits no local-in-time solution with $u_0$ defined as $\eqref{Nonexist-u_0def}$. Assume by contradiction that, problem $\eqref{1.1}$ admits a solution in $\RN\times(0,T)$ for some $T>0$. 
  Then, by Lemma \ref{Nonexist-lem}, we have 
  \begin{equation}\label{Non1}
    [e^{t\Delta} u_0](0) \le F^{-1}\left(\frac{1}{C} t^{1-\frac{\gamma}{2}}\right) \quad \text{for}~~t \in (0, T). 
  \end{equation}
  On the other hand, since
  \begin{equation*}
    \frac{d^{2}}{d\xi^{2}} \left[F^{-1}(\xi) \right] = f\left(F^{-1}(\xi) \right) f'\left(F^{-1}(\xi) \right)\ge0 \quad \text{for}~~\xi>0, 
  \end{equation*}
  $F^{-1}$ is convex on $(0, \infty)$. Thus, it follows from Proposition~\ref{Jensen} that
  \begin{equation}\label{Non2}
    \begin{split}
      [e^{t\Delta} u_0](0) &= \int_{\RN} G(y, t) F^{-1}(|y|^{A}) \,dy \\
      &\ge F^{-1}\left(\int_{\RN} G(y, t) |y|^{A} \,dy \right) \\
      &= F^{-1}\left((4\pi t)^{-\frac{N}{2}} \int_{\RN} e^{-\frac{|y|^{2}}{4t}} |y|^{A} \,dy \right) \\
       &=F^{-1}(C t^{\frac{A}{2}}) 
     \end{split}
  \end{equation}
  for $t\in(0,T)$. Thus, by $\eqref{Non1}$ and $\eqref{Non2}$, we have
  \begin{equation*}
    F^{-1}\left(C t^{\frac{A}{2}}\right) \le F^{-1}\left(\frac{1}{C} t^{1-\frac{\gamma}{2}} \right) 
  \end{equation*}
  for $t\in(0,T)$. Since $F^{-1}$ is decreasing, we have
  \begin{equation*}
    t^{\frac{A}{2}} \ge C t^{1-\frac{\gamma}{2}}, \quad \text{that is,} \quad t^{\frac{A}{2} - \left(1-\frac{\gamma}{2} \right)} \ge C 
  \end{equation*}
  for $t \in (0, T)$. Therefore, this is contradiction for small $t>0$, since $A/2 - (1-\gamma/2)>0$. This completes the proof of Theorem~\ref{Nonexist-thm}. 
\end{proof}

\subsection*{Acknowledgements}
The author would like to thank Professor Jun-ichi Segata for useful advice and suggestions. He would also like to thank Assistant Professor Nobuhito Miyake for helpful discussions and comments.


\begin{thebibliography}{99}
\bibitem{ACCL25} A.~Aparcana, B.~Carhuas-Torre, R.~Castillo and M.~Loayza, 
\textit{Existence and non-existence of solutions for {H}ardy parabolic equations with singular initial data}, 
Electron. J. Differential Equations (2025), Paper No. 67, 11. 

\bibitem{ARCD04} J.M.~Arriera, A.~Rodriguez-Bernal, J.W.~Cholewa and T.~Dlotko,
\textit{Linear parabolic equations in locally uniform spaces}, 
Math. Models Methods Appl. Sci. {\bf 14} (2004), 253-293. 

\bibitem{BC96} H.~Brezis and T.~Cazenave, 
\textit{A nonlinear heat equation with singular initial data}, 
J. Anal. Math. {\bf 68} (1996), 277--304. 

\bibitem{BTW17} B.~Ben Slimene, S.~Tayachi and F.B.~Weissler, 
\textit{Well-posedness, global existence and large time behavior for {H}ardy-{H}\'enon parabolic equations}, 
Nonlinear Anal. {\bf 152} (2017), 116--148. 

\bibitem{CCL25} B.~Carhuas-Torre, R.~Castillo and M.~Loayza, 
\textit{The {H}ardy parabolic problem with initial data in uniformly local {L}ebesgue spaces}, 
J. Differential Equations {\bf 424} (2025), 438--462. 

\bibitem{CGL22} R.~Castillo, O.~Guzm\'an-Rea and M.~Loayza, 
\textit{On the local existence for {H}ardy parabolic equations with singular initial data}, 
J. Math. Anal. Appl. {\bf 510} (2022), 126022. 

\bibitem{C19} N.~Chikami, 
\textit{Composition estimates and well-posedness for {H}ardy-{H}\'enon parabolic equations in {B}esov spaces}, 
J. Elliptic Parabol. Equ. {\bf 5} (2019), 215--250. 

\bibitem{CIT21} N.~Chikami, M.~Ikeda and K.~Taniguchi, 
\textit{Well-posedness and global dynamics for the critical {H}ardy-{S}obolev parabolic equation}, 
Nonlinearity {\bf 34} (2021), 8094--8142. 

\bibitem{CIT22} N.~Chikami, M.~Ikeda and K.~Taniguchi, 
\textit{Optimal well-posedness and forward self-similar solution for the {H}ardy-{H}\'enon parabolic equation in critical weighted {L}ebesgue spaces}, 
Nonlinear Anal. {\bf 222} (2022), Paper No. 112931, 28. 

\bibitem{F14} Y.~Fujishima, 
\textit{Blow-up set for a superlinear heat equation and pointedness of the initial data}, 
Discrete Contin. Dyn. Syst. {\bf 34} (2014), 4617--4645. 

\bibitem{FHIL24} Y.~Fujishima, K.~Hisa, K.~Ishige and R.~Laister, 
\textit{Local solvability and dilation-critical singularities of supercritical fractional heat equations}, 
J. Math. Pures Appl. (9) {\bf 186} (2024), 150--175. 

\bibitem{FI18} Y.~Fujishima and N.~Ioku, 
\textit{Existence and nonexistence of solutions for the heat equation with a superlinear source term}, 
J. Math. Pures Appl. (9) {\bf 118} (2018), 128--158. 

\bibitem{FI21} Y.~Fujishima and N.~Ioku, 
\textit{Solvability of a semilinear heat equation via a quasi scale invariance}, 
Geometric properties for parabolic and elliptic {PDE}s, Springer INdAM Ser., {\bf 47} Springer, Cham, [2021] \copyright 2021, 79--101. 

\bibitem{FI23} Y.~Fujishima and N.~Ioku, 
\textit{Quasi self-similarity and its application to the global in time solvability of a superlinear heat equation}, 
Nonlinear Anal. {\bf 236} (2023), 113321. 

\bibitem{G86} Y.~Giga, 
\textit{Solutions for semilinear parabolic equations in {$L^p$} and regularity of weak solutions of the {N}avier-{S}tokes system}, 
J. Differential Equations {\bf 62} (1986), 186--212. 

\bibitem{GM22} T.~Giraudon and Y.~Miyamoto, 
\textit{Fractional semilinear heat equations with singular and nondecaying initial data}, 
Rev. Mat. Complut. {\bf 35} (2022), 415--445. 

\bibitem{HS24} K.~Hisa and M.~Sier\.z\c ega,
\textit{Existence and nonexistence of solutions to the {H}ardy parabolic equation}, 
Funkcial. Ekvac. {\bf 67} (2024), 149--174. 

\bibitem{HT21} K.~Hisa and J.~Takahashi, 
\textit{Optimal singularities of initial data for solvability of the {H}ardy parabolic equation}, 
J. Differential Equations {\bf 296} (2021), 822--848. 

\bibitem{HM25} K.~Hisa and Y.~Miyamoto, 
\textit{Threshold property of a singular stationary solution for semilinear heat equations with exponential growth}, 
Manuscripta Math. {\bf 176} (2025), Paper No. 61, 30. 

\bibitem{LRSL16} R.~Laister, J.C.~Robinson, M.~Sier\.z\c ega and A.~Vidal-L\'opez, 
\textit{A complete characterisation of local existence for semilinear heat equations in {L}ebesgue spaces}, 
Ann. Inst. H. Poincar\'e{} C Anal. Non Lin\'eaire {\bf 33} (2016), 1519--1538. 

\bibitem{MT06} Y.~Maekawa and Y.~Terasawa,
\textit{The {N}avier-{S}tokes equations with initial data in uniformly local {$L^p$} spaces}, 
Differential Integral Equations {\bf 19} (2006), 369--400. 


\bibitem{MS21} Y.~Miyamoto and M.~Suzuki, 
\textit{Thresholds on growth of nonlinearities and singularity of initial functions for semilinear heat equations}, 
arXiv:2104.14773. 

\bibitem{P97} R.G.~Pinsky, 
\textit{Existence and nonexistence of global solutions for {$u_t=\Delta u+a(x)u^p$} in {${\bf R}^d$}}, 
J. Differential Equations {\bf 133} (1997), 152--177. 

\bibitem{RS13} J.C.~Robinson and M.~Sier\.z\c ega, 
\textit{Supersolutions for a class of semilinear heat equations}, 
Rev. Mat. Complut. {\bf 26} (2013), 341--360. 

\bibitem{T20} S.~Tayachi, 
\textit{Uniqueness and non-uniqueness of solutions for critical {H}ardy-{H}\'enon parabolic equations}, 
J. Math. Anal. Appl. {\bf 488} (2020), 123976, 51. 

\bibitem{W93} X.~Wang, 
\textit{On the {C}auchy problem for reaction-diffusion equations}, 
Trans. Amer. Math. Soc. {\bf 337} (1993), 549--590. 

\bibitem{W79} F.B.~Weissler, 
\textit{Semilinear evolution equations in {B}anach spaces}, 
J. Functional Analysis {\bf 32} (1979), 277--296. 

\bibitem{W80} F.B.~Weissler, 
\textit{Local existence and nonexistence for semilinear parabolic equations in {$L\sp{p}$}}, 
Indiana Univ. Math. J. {\bf 29} (1980), 79--102. 

\bibitem{W81} F.B.~Weissler, 
\textit{Existence and nonexistence of global solutions for a semilinear heat equation}, 
Israel J. Math. {\bf 38} (1981), 29--40. 
  
\end{thebibliography}
\end{document}